\documentclass[11pt]{article}

\usepackage{amsmath, amsthm, amssymb}  %all AMS (amssymb loads amsfonts)
\usepackage{mathrsfs,enumitem}
\usepackage{mathabx,bbm}

\numberwithin{equation}{section}

{\theoremstyle{plain}

\newtheorem{lemma}{Lemma}[section]
\newtheorem{theorem}[lemma]{Theorem}

} {\theoremstyle{definition}

}

\renewcommand{\mathbb}{\mathbbm}                     % use mathbbm
\renewcommand{\epsilon}{\varepsilon}                 % AMS symbols
\renewcommand{\phi}{\varphi}
\renewcommand{\theta}{\vartheta}
\renewcommand{\le}{\leqslant}
\renewcommand{\ge}{\geqslant}
\newcommand{\origsetminus}{} \let\origsetminus=\setminus     % redefine setminus
\renewcommand{\setminus}{\!\origsetminus\!}
\newcommand{\origfoo}{} \let\origfoo=\sqrt           % redefine square root
\renewcommand{\sqrt}[1]{\origfoo{#1}\;}

\renewcommand{\O}{{\mathcal O}}                      % big Landau
\newcommand{\abs}[1]{\left\lvert #1 \right\rvert}    % absolut value
\newcommand{\norm}[1]{\left\lVert #1 \right\rVert}   % norm
\DeclareMathOperator{\F}{{\mathcal F}}                   % sigma-fiel
\DeclareMathOperator{\R}{{\mathbb R}}                % reals
\DeclareMathOperator{\Rp}{{\mathbb R}_+}             % positive reals
\DeclareMathOperator{\N}{{\mathbb N}}                % integer
\DeclareMathOperator{\Borel}{{\mathcal B}}
\newcommand{\scapro}[2]{\langle #1,#2\rangle}       %Skalarprodukt
\newcommand{\scaprob}[2]{\big\langle #1,#2\big\rangle}       %Skalarprodukt
\newcommand{\scaprobb}[2]{\Big\langle #1,#2\Big\rangle}

\DeclareMathOperator{\1}{\mathbbm 1}
\newcounter{zahl}

\DeclareMathOperator{\Z}{{\mathcal Z}}

\numberwithin{equation}{section}

\newcommand{\bcase}{\begin{cases}}
\newcommand{\ecase}{\end{cases}}
\newcommand{\pmat}{\begin{pmatrix}}
\newcommand{\epmat}{\end{pmatrix}}

\newcommand{\barray}{\begin{array}{rcl}}
\newcommand{\earray}{\end{array}}

\newcommand{\del}[1]{}

\renewcommand{\div}{{\;\mbox{div}\; }}

\newcommand{\be} {\begin{enumerate} }
\newcommand{\ee} {\end{enumerate} }

\newcommand{\TT}{{\rm I \kern -0.2em T}}

\newcommand{\DEQS}{\begin{eqnarray*}}
\newcommand{\EEQS}{\end{eqnarray*}}
\newcommand{\DEQSZ}{\begin{eqnarray}}
\newcommand{\EEQSZ}{\end{eqnarray}}

\allowdisplaybreaks

\begin{document}

\title{Large deviations for stochastic heat equations  with memory driven 
by L\'evy-type noise}

\author{
Markus Riedle\\
Department of Mathematics\\
King's College London\\
London WC2R 2LS\\
United Kingdom\\
markus.riedle@kcl.ac.uk
\and
Jianliang Zhai\footnote{The second author acknowledges 
funding by K. C. Wong Foundation for a 1-year fellowship at King's College London}\\
School of Mathematical Sciences\\University of Science \\and Technology of China\\
 Chinese Academy of Science\\
Hefei, 230026, China\\
zhaijl@ustc.edu.cn}

\maketitle

\begin{abstract}
For a heat equation with memory driven by a L\'evy-type noise we establish the existence of a unique solution. The main part of the article focuses on the Freidlin-Wentzell large deviation principle of the  solutions of heat equation with memory driven by a L\'evy-type noise. For this purpose, we exploit the recently introduced
weak convergence approach.
\end{abstract}


 \thispagestyle{empty}

\section{Introduction}

In this work we consider a non-linear heat equation with memory driven by a L\'evy-type noise. Heat equations with memory have been
considered for a long time and their study has recently been published in the monograph \cite{Amendola-etal}. 
In order to correct the non-physical property of instantaneous propagation for the heat equation, Gurtin and Pipkin introduced in 
\cite{GurtinPipkin}  a modified Fourier's law,  which resulted in a heat equation with memory.  More precisely, let $u(t,x)$ denote 
the temperature at time $t$ at position $x$ in a bounded domain $\bar{\mathcal{O}}$.
By following the theory  developed in \cite{Coleman-Gurtin}, \cite{GurtinPipkin} and \cite{Nunziato}, the temperature $u(t,x)$ 
and the density $e(t,x)$ of the internal energy and the heat flux $\phi(t,x)$ are related by
\begin{align*}
e(t,x)&=e_0+b_0u(t,x)\qquad \text{ for } t\in\mathbb{R}^+,\, x\in\bar{\mathcal{O}},\\
\phi(t,x)&=-c_0\nabla u(t,x)+\int_{-\infty}^t \gamma(r)\nabla u(t+r,x)\,dr\qquad \text{ for }t\in\mathbb{R}^+,\ \ x\in\bar{\mathcal{O}}.
\end{align*}
Here, the constant $e_0$ denotes the internal energy at equilibrium and the constants $b_0>0$ and $c_0>0$ are the heat capacity and thermal conduction. The heat flux relaxation is described by the function $\gamma\colon (-\infty,0]\to \Rp$. The energy
 balance for the system has the form
 $$
 \partial_t e(t,x)=-\text{div}\phi(t,x)+\widetilde{F}(t,u(t,x)),
 $$
 where $\widetilde{F}$ is the nonlinear heat supply which might describe temperature-dependent radiative phenomena. After rescaling the constants and generalizing $\widetilde{F}$, we arrive at
 \begin{align}\label{eq.det}
\partial_t u(t,x)&=\Delta u(t,x)+ F\big(t,u(t,x)\big)
 + \!\! \div B\big(t,u(t,x)\big)+
 \int_{-\infty}^0 \!\!\gamma(r)\Delta u(t+r,x)\, dr.
\end{align}
This equation models the heat flow in a rigid, isotropic, homogeneous heat conductor
with linear memory and is considered for example in \cite{GiorgiPata}, \cite{GiorgiPataMarzocchi} and  \cite{GurtinPipkin}.

It is well known that many physical phenomena are better described
by taking into account some kind of uncertainty, for instance some randomness or random environment. For this purpose, we assume in this work that the heat supply $\tilde{F}$ in the derivation above contains a stochastic term representing an environmental noise. In order to accommodate a general non-Gaussian environmental noise with possibly discontinuous trajectories, we model the noise by a L\'evy-tpe 
stochastic process. Repeating the above derivation with $\widetilde{F}$ containing 
such random noise, we arrive at a stochastically perturbed version of Equation 
\eqref{eq.det}; see Equation \eqref{eq heat} in the next section.  The main aim of our work is to establish the existence and uniqueness theorem and to investigate the Freidlin-Wentzell large deviation principle of the  solutions of the stochastically perturbed equation. 

For the case of a Gaussian environmental noise, there exists a great amount of literature. For instance, results on the well-posedness of the resulting equation  were obtained in  \cite{BarbuBonaccorsiTubaro}, \cite{Caraballo-etal-07} and \cite{LXZ}. Long time behaviors of the stochatically perturbed equation for a Gaussian environmental  noise were studied in \cite{BarbuBonaccorsiTubaro}, \cite{BonaccorsiDaPratoTubaro},  \cite{Caraballo-etal-07} and \cite{CaraballoRealChueshov}. The Freidlin-Wentzell large deviation principle in this situation had been studied in \cite{LXZ} under certain restriction on the  the nonlinearity. 

In our case of a L\'evy-type environmental noise, we will obtain the well-posedness
of the stochastically perturbed equation by the classical cutting-off method. 
Due to the appearance of jumps in our setting, the Freidlin-Wentzell large deviation principle are distinctively different to the Gaussian case in \cite{LXZ}. We will use the weak convergence approach introduced in \cite{Budhiraja-Chen-Dupuis} and \cite{Budhiraja-Dupuis-Maroulas.} for the case of Poisson random measures. This approach is a powerful tool to establish the Freidlin-Wentzell large deviation principle for various finite and infinite dimensional stochastic dynamical systems with irregular coefficients driven by a non-Gaussian 
L\'evy noise, see  for example \cite{Bao-Yuan}, \cite{Budhiraja-Chen-Dupuis}, \cite{Dong-Xiong-Zhai-Zhang}, \cite{Xiong-Zhai}, \cite{YZZ} and \cite{Zhai-Zhang}. %Following the weak convergence method, Bao and Yuan \cite{Bao-Yuan} obtained a (uniform) Freidlin-Wentzell's large deviation principle for a class of stochastic functional differential equations of neutral type driven by a finite-dimensional Wiener process and a Poisson random measure.
 The main point of our approach is to prove the tightness of some controlled stochastic dynamical systems. For this purpose, we exploit estimates of the controlled stochastic dynamical systems which significantly differ from those in the Gaussian setting.

The organization of this paper is as follows. In Section 2, we introduce the assumptions and establish the well-posedness of the stochastically perturbed 
equation. Section 3 is devoted to establishing the
Freidlin-Wentzell's large deviation principle.

\section{Existence of a solution}

Let $(\Omega,\; \mathcal{F},\;\mathbb{F}:=\{\mathcal{F}_t\}_{t\in[0,T]},\; P)$ be a filtered probability space with an adapted,  standard cylindrical Brownian motion $W$ on a Hilbert space $H$. 
Let $\widetilde{N}$ be a compensated time homogeneous Poisson random measure on a Polish space $X$ with intensity $\nu$ and assume $\widetilde{N}$ to be independent of $W$. If $S$ is a separable metric space we denote the space of equivalence classes of random variables $Vc\colon \Omega\to S$ by $L^0(\Omega, S)$ and the Borel $\sigma$-algebra by $\Borel(S)$. 

Consider the following stochastic heat equation on $L^2:=L^2(\O)$ for a bounded domain $\O\subseteq \R^d$
for $t\in [0,T]$:
\begin{align}\label{eq heat}
 U(t)&=u_0+\int_0^t \left( \Delta U(s)+ F\big(s,U(s)\big)
 + \div B\big(s,U(s)\big)+
 \int_{-\infty}^0 \gamma(r)\Delta U(s+r)\, dr
 \right)\,ds \nonumber\\
 &\qquad
 + \int_0^t G_1\big(s,U(s)\big)\,dW(s)+ \int_0^t \int_X G_2\big(s,U(s-),x\big)\,\widetilde{N}(ds,dx),\\
 U(s)&=\varrho(s) \quad\text{for }s\le 0,\nonumber
\end{align}
where the initial condition is given by $u_0\in L^2$ and the square Bochner integrable function $\varrho\colon (-\infty,0]\to L^2$. The past dependence is described by the function $\gamma\in L^1(\R_{-},\Rp)$. The non-linear drift is described by
the operator $F\colon [0,T]\times L^2\to L^2$. The diffusion coefficients
are described by the operators $G_1\colon [0,T]\times L^2\to {\mathcal{L}_2}$ and
$G_2\colon  [0,T]\times L^2\times X\to L^2$, where ${\mathcal{L}_2}$ denotes the space of Hilbert-Schmidt operators from $H$ to $L^2$. 

We begin with introducing some notations. For $p\geq 1$, we denote by $L^p(\mathcal{O})$ the usual $L^p$-space over $\mathcal{O}$ with the standard norm $\|\cdot\|_p$.
For $m\in\mathbb{N}$, let $H^m_0(\mathcal{O})$ be the usual $m$-order Sobolev space over $\mathcal{O}$ with Dirichlet boundary conditions, and denote its
norm and the dual space by $\|\cdot\|_{2,m}$ and $H^{-m}(\mathcal{O})$, respectively. For  simplicity, we will write
$$
L^p:=L^p(\mathcal{O}),\ \ H^m_0:=H^m_0(\mathcal{O}),\ \ H^{-m}:=H^{-m}(\mathcal{O}).
$$
Sobolev's embedding theorem (see e.g. \cite{Adams}) guarantees for any $q\geq 2$ and $q^*:=q/(q-1)$ that
$H^d_0\hookrightarrow L^q$ and $L^{q^*}\hookrightarrow H^{-d}$.
It is well known that the Laplacian $\Delta$ establishes an isomorphism between $H^1_0$ and $H^{-1}$. Since $H^1_0$ coincides with the domain of the operator $(-\Delta)^{1/2}$,
we will use the following equivalent norm in $H^1_0$:
$$
\|u\|_{2,1}:=\|\nabla u\|_2=\|(-\Delta)^{1/2}u\|_2.
$$
%The operator $-\Delta$ is positive selfadjoint with compact resolvent; we denote by $0<\lambda_1\leq \lambda_2\leq \cdots\uparrow\infty$ the eigenvalues of $-\Delta$, and by $e_1,e_2,\cdots$ a corresponding
%complete orthonormal system in $L^2$ of eigenvectors of $-\Delta$.
Notice that there exists a constant $\lambda_1>0$ such that
$\lambda_1\|u\|_2^2\leq \|\nabla u\|_2^2$ for all $u\in H^1_0$.
For $q\geq 2$ we define $V_q:=H^1_0\cap L^q$ and $V^*_q:=H^{-1}+L^{q^*}$.
By identifying $L^2$ with itself by the Riesz representation, we obtain an evolution triple
$$
V_q\subset L^2\subset V^{*}_q.
$$
That is, for any $v\in V_q$ and $w=w_1+w_2\in H^{-1}+L^{q^*}$ we have
$$
\langle v,w\rangle_{V_q,V_q^*}=\langle v,w_1\rangle_{H^1_0,H^{-1}}+\langle v,w_2\rangle_{L^q,L^{q^*}}.
$$
For the simplicity of notation, when no confusion may arise, we will use the unified notation $\langle \cdot,\cdot\rangle$ to denote the above dual relations between different spaces.

Denote the Lebesgue measures on $[0,T]$ and $[0,\infty)$ by Leb$_T$ and Leb$_\infty$, respectively and define $\nu_T:={\rm Leb}_T\otimes\nu$. We introduce the function space
\begin{align*}
\mathcal{H}:=\Big\{h\colon [0,T]\times X&\rightarrow \mathbb{R}:\;\int_{\Gamma}\exp(\delta h^2(t,x))\,\nu(dx)\,dt<\infty,\\
&\text{for all }\Gamma\in \mathcal{B}([0,T]\times X) \text{ with }\nu_T(\Gamma)<\infty
\text{ and for some }\delta>0     \Big\}.
\end{align*}

Throughout this work we will assume the following assumptions:
\begin{enumerate}
\item[H1:] there exist constants $c_1$, $c_2>0$ and $h_1\in L^1([0,T],\Rp)$ such that
we have for all $v_1$, $v_2\in L^2$ and $t\in [0,T]$:
\begin{align*}
\norm{B(t,v_1)-B(t,v_2)}_2^2&\le c_1 \norm{v_1-v_2}_2^2,\\
\norm{B(t,v_1)}_2^2 &\le c_2 \norm{v_1}_2^2+ h_1(t).
\end{align*}
\item[H2:] there exist constants $c_3$, $c_4$, $c_5>0$, $q\ge 2$ and $h_2$, $h_3\in L^1([0,T],\Rp)$
 such that
we have for all $v_1$, $v_2\in L^2$ and $t\in [0,T]$:
\begin{align}
\langle v_1-v_2,F(t,v_1)-F(t,v_2)\rangle &\le c_3 \norm{v_1-v_2}_2^2, \\
\scapro{v_1}{F(t,v_1)}&\le -c_4 \norm{v_1}_{L^q }^q + h_2(t)\big(1+ \norm{v_1}_2^2 \big),\label{eq.scapro-g}\\
\norm{F(t,v_1)}_{L^{q^*} }^{q^*} &\le c_5\norm{v_1}_{L^q }^{q}+ h_3(t),\\
\text{for any } x\in L^q ,\ y,z\in H^1_0 ,&\text{ the mapping }\nonumber \\
 \eta\mapsto \langle x, &F(t,y+\eta z)\rangle_{L^q ,L^{q^*} }\text{ is continuous on $[0,1]$.}
\end{align}
\item[H3:] there exist $c_6>0$, $h_4\in L^1([0,T],\Rp)$ such that
we have for all $v_1$, $v_2\in L^2$ and $t\in [0,T]$:
\begin{align}
\norm{G_1(t,v_1)-G_1(t,v_2)}^2_{{\mathcal{L}_2}}&\le c_6\norm{v_1-v_2}_2^2,\label{eq.sigma_1-lipschitz}\\
\norm{G_1(t,v_1)}^2_{{\mathcal{L}_2}}&\le h_4(t)\big(1+  \norm{v_1}_2^2\big).
 \label{eq.sigma_1-growth}
\end{align}
\item[H4:]  there exist $h_5, h_6\in L^2_{\nu_T}([0,T]\times X, \R)\cap \mathcal{H}$  such that
we have for all $v_1$, $v_2\in L^2$, $x\in X$ and $t\in [0,T]$:
\begin{align}
%\int_X
\norm{G_2(t,v_1,x)-G_2(t,v_2,x)}_2&\le h_5(t,x)\norm{v_1-v_2}_2, \label{eq.sigma_2-lipschitz}\\
%\int_X
\norm{G_2(t,v_1,x)}_2&\le h_6(t,x)\big(1+\norm{v_1}_2\big).
  \label{eq.sigma_2-growth}
\end{align}
\end{enumerate}

Our main theorem in this section guarantees the existence of a unique solution of Equation \eqref{eq heat}.
\begin{theorem}\label{thm solution}
Assume (H1)-(H4). Then for every $(u_0,\varrho)\in L^2\times  L^2(\mathbb{R}_{-}, H^1_0 )$, there exists a unique $\mathbb{F}$-adapted stochastic process $U\in L^0\big(\Omega,\, D([0,T],L^2)\cap L^2([0,T],H^1_0 )\big)$ satisfying Equation 
\eqref{eq heat} and 
$$
E\left[\sup_{t\in[0,T]}\|U(t)\|^2_2+ \int_0^T\|U(t)\|^2_{2,1 }ds+\int_0^T\|U(t)\|^q_q \, dt\right]<\infty.
$$
\end{theorem}
%
%\begin{remark}
%To obtain Theorem \ref{thm solution}, we only need $h_5, h_6$ in (H4) belongs to $L^2(\nu_T)$.
%\end{remark}
%
%

\begin{proof}
As the proof is rather standard, see for example \cite{LXZ}, we suppress some details.  \\
\noindent {\bf Step 1.}
As $\nu$ is $\sigma$-finite on the Polish space $X$, there exist measurable subsets $K_m\uparrow X$ satisfying $\nu(K_m)<\infty$ for all $m\in\N$. 
%Define for arbitrary differentiable functions
%$h\colon [0,T]\to L^2$ and $t\in [0,T]$ the operator
%\begin{align*}
%r(t,h):= \div B\big(t,h(t)\big) + \int_{-t}^ 0 \gamma(l)\Delta h(t+l)\, dl 
% +\int_{-\infty}^{-t} \gamma(l)\Delta \varrho(t+l)\,dl .
%\end{align*}
For each $m\in\N$ define the function $U_{m,0}\equiv 0$ and consider recursively
the equations
\begin{align}\label{eq.Picard}
U_{m,n}(t):=u_0 &+ \int_0^t \bigg(\Delta U_{m,n}(s) + F\big(s,U_{m,n}(s)\big) \notag \\
&\qquad + \div B\big(s,U_{m,{n-1}}(s)\big)+ \int_{-\infty}^0 \gamma(r)\Delta U_{m,{n-1}}(s+r)\, dr \bigg)\,ds\\
& + \int_0^t G_1\big(s,U_{m,n}(s)\big)\,dW(s)
+ \int_0^t \int_{K_m} G_2\big(s,U_{m,n}(s-),x\big)\,\widetilde{N}(ds,dx). \notag
\end{align}
As in the proof of \cite[Th.3.2]{LXZ} it follows that  
there exists an $\mathbb{F}$-adapted stochastic process $U_{m,n}\in L^0\big(\Omega,\, D([0,T],L^2)\cap L^2([0,T],H^1_0)\big)$ satisfying \eqref{eq.Picard} for each $m,n\in\N$.

In a first step we show that
there exists $T_0>0$ and a constant $C>0$ such that
for all $m$, $n\in\N$ we have:
\begin{align}
E\left[\sup_{t\in[0,T_0]}\norm{U_{m,n}(t)}_2^2
+ \int_0^{T_0}\norm{U_{m,n}(s)}_{2,1 }^2\,ds
+\int_0^{T_0}\norm{U_{m,n}(s)}_{L^q }^q\,ds \right]\le C.\label{eq u mn}
\end{align}
For this purpose, we apply It{\^o}'s formula to obtain
\begin{align}
\norm{U_{m,n}(t)}_2^2 
&= \norm{u_0}_2^2 -2 \int_0^t \norm{U_{m,n}(s)}_{2,1 }^2 \,ds\notag\\
&\quad+ 2\int_0^t \Big\langle U_{m,n}(s),\,\ F(s,U_{m,n}(s)) \notag \\
&\quad\qquad  + \div B\big(s,U_{m,{n-1}}(s)\big)+ \int_{-\infty}^0 \gamma(r)\Delta U_{m,{n-1}}(s+r)\, dr\Big\rangle \,ds \notag\\
&\quad +2 \int_0^t \scapro{U_{m,n}(s)}{G_1\big(s,U_{m,n}(s)\big)}
\,dW(s)+ \int_0^t \norm{G_1(s,U_{m,n}(s))}^2_{{\mathcal{L}_2}}\,ds\notag \\
&\quad +
2\int_0^t \int_{K_m} \scapro{U_{m,n}(s-)}{G_2\big(s,U_{m,n}(s-),x\big)}\,\widetilde{N}(ds,dx)\notag\\
&\quad + \int_0^t \int_{K_m} \norm{G_2\big(s,U_{m,n}(s-),x\big)}_2^2\,N(ds,dx).
\label{eq.ito}
\end{align}
It follows from \cite[Le.3.1]{LXZ} and Cauchy-Schwartz inequality 
that there exists a constant $a>0$ such that
\begin{align}
& 2 \int_0^t \abs{\scaprobb{U_{m,n}(s)}{\div B\big(s,U_{m,{n-1}}(s)\big)+ \int_{-\infty}^0 \gamma(r)\Delta U_{m,{n-1}}(s+r)\, dr}}\,ds\notag \\
&\quad \le \int_0^t \norm{U_{m,n}(s)}_{2,1 }^2ds +2\delta_t^2\int_0^t\norm{U_{m,n-1}(s)}_{2,1 }^2ds+ a +a\int_0^t\norm{U_{m,n-1}(s)}_2^2\,ds,\label{eq.aux1.P4}
\end{align}
where  $\delta_t:=\int_{-t}^0|\gamma(s)|ds$.
Assumption \eqref{eq.scapro-g} guarantees that
\begin{align}
&  \int_0^t \scapro{U_{m,n}(s)}{F(s,U_{m,n}(s))}\,ds\notag \\
&\qquad
\le -c_4 \int_0^t \norm{U_{m,n}(s)}_{L^q }^qds + \int_0^t h_2(s)\big(1+ \norm{U_{m,n}(s)}_2^2\big)\, ds,
\end{align}
and Assumption \eqref{eq.sigma_1-growth} guarantees that
\begin{align}
\int_0^t \norm{G_1(s,U_{m,n}(s))}^2_{{\mathcal{L}_2}}\,ds
\le \int_0^t h_4(s)\big(1+  \norm{U_{m,n}(s)}_2^2\big)ds.\label{eq.aux2}
\end{align}
By applying \eqref{eq.aux1.P4} - \eqref{eq.aux2} to equation \eqref{eq.ito} we obtain
\begin{align}\label{eq 2.16}
&\norm{U_{m,n}(t)}_2^2 +\int_0^t \norm{U_{m,n}(s)}_{2,1 }^2\,ds
  + c_4 \int_0^t \norm{U_{m,n}(s)}_{L^q }^q\,ds \notag\\
&\quad \le \norm{u_0}_2^2 +
\int_0^t \big(2h_2(s)+h_4(s)\big)\big(1+\norm{U_{m,n}(s)}_2^2\big)\,ds \notag\\
&\qquad +2\delta_t^2\int_0^t\norm{U_{m,n-1}(s)}_{2,1 }^2ds+ a +a\int_0^t\norm{U_{m,n-1}(s)}_2^2\,ds \notag\\
&\qquad +2 \int_0^t \scapro{U_{m,n}(s)}{G_1\big(s,U_{m,n}(s)\big)}
\,dW(s)\notag\\
&\qquad +2\int_0^t \int_{K_m} \scapro{U_{m,n}(s-)}{G_2\big(s,U_{m,n}(s-),x\big)}\,\widetilde{N}(ds,dx)\notag\\
&\qquad + \int_0^t \int_{K_m} \norm{G_2\big(s,U_{m,n}(s-),x\big)}_2^2\,N(ds,dx).
\end{align}
By applying Burkholder's inequality and Young's inequality, we obtain that for every $r>0$ and $\kappa_1\in (0,1)$ there exists a constant
$c_{\kappa_1}>0$ such that
%\begin{align*}
%&E\left[\sup_{t\in [0,r]}\abs{\int_0^t \scapro{U_{m,n}(s)}{\sigma_1\big(s,U_{m,n}(s)\big)}\,dW(s)}^2\right]\\
%&\qquad \le E\left[\int_0^r \norm{U_{m,n}(s)}_V^2\norm{\sigma_1\big(s,U_{m,n}(s)\big)}_{L_2(U,V)}^2\,ds\right]\\
%&\qquad \le E\left[\sup_{s\in [0,r]}\norm{U_{m,n}(s)}_V^2\int_0^r \norm{\sigma_1\big(s,U_{m,n}(s)\big)}_{L_2(U,V)}^2\,ds\right].
%\end{align*}
%It follows from  that for every
\begin{align}
&E\left[\sup_{t\in [0,r]}\abs{\int_0^t \scapro{U_{m,n}(s)}{G_1\big(s,U_{m,n}(s)\big)}\,dW(s)}\right]\notag \\
&\qquad\le \kappa_1 E\left[\sup_{s\in [0,r]}\norm{U_{m,n}(s)}_2^2\right]
+ c_{\kappa_1} E\left[\int_0^r \norm{G_1(s,U_{m,n}(s)}_{{\mathcal{L}_2}}^{2}\,ds\right]\notag\\
&\qquad \le  \kappa_1 E\left[\sup_{s\in [0,r]}\norm{U_{m,n}(s)}_2^2\right]
+ c_{\kappa_1}  E\left[\int_0^r h_4(s)\big(1+ \norm{U_{m,n}(s)}_2^2\big) \,ds\right],\label{eq.kappa1}
\end{align}
where we used Assumption \eqref{eq.sigma_1-growth} in the last equality.
Similarly, we conclude from Burkholder's inequality, Young's inequality and Assumption \eqref{eq.sigma_2-growth}, that for each $r>0$ and $\kappa_2\in (0,1)$ there exists a constant $c_{\kappa_2}>0$ such that
%\begin{align*}
%&E\left[\sup_{t\in [0,r]}\abs{\int_0^t \int_{K_m} \scapro{U_{m,n}(s-)}{\sigma_2\big(s,U_{m,n}(s-),x\big)}\,\widetilde{N}(ds,dx)}^2\right]\\
%&\qquad \le c
%E\left[\int_0^r\int_{K_m} \norm{u_{m,n}(s)}_V^2 \norm{\sigma_2\big(s,u_{m,n}(s),x\big)}_{2}^2\,\nu(dx)\, ds\right]\\
%&\qquad \le E\left[\sup_{s\in [0,r]}\norm{u_{m,n}(s)}_V^2 \int_0^r\int_{K_m}  \norm{\sigma_2\big(s,u_{m,n}(s),x\big)}_{2}^2\,\nu(dx)\, ds \right].
%\end{align*}
%As above,  imply
%that for every
\begin{align}
&E\left[\sup_{t\in [0,r]}\abs{\int_0^t \int_{K_m} \scapro{U_{m,n}(s-)}{G_2\big(s,U_{m,n}(s-),x\big)}\,\widetilde{N}(ds,dx)}\right]\notag \\
%&\qquad \le \kappa_2 E\left[\sup_{s\in [0,r]}\norm{U_{m,n}(s)}_2^2\right] + c_{\kappa_2} E\left[ \int_0^r%\int_{K_m}  \norm{G_2\big(s,U_{m,n}(s),x\big)}_{2}^2\,\nu(dx)\, ds \right]\notag\\
&\qquad \le \kappa_2 E\left[\sup_{s\in [0,r]}\norm{U_{m,n}(s)}_2^2\right] + c_{\kappa_2} E\left[ \int_0^r
  h_6(s)\big(1+ \norm{U_{m,n}(s)}_2^2\big)\, ds \right], \label{eq.kappa2}
\end{align}
with $h_6(s):=\int_{X}h_6^2(s,x)\,\nu(dx)$ for all $s\in [0,T]$.
Another application of Assumption \eqref{eq.sigma_2-growth} implies for each $r\in [0,T]$ that 
\begin{align}
&E\left[\int_0^r \int_{K_m} \norm{G_2\big(s,U_{m,n}(s-),x\big)}_2^2\,N(ds,dx)\right]\notag\\
%&\qquad = E\left[\int_0^r \int_{K_m} \norm{G_2\big(s,U_{m,n}(s),x\big)}_2^2 \, \nu(dx)\,ds\right]\notag\\
&\qquad\qquad\qquad  \le E\left[ \int_0^r  h_6(s)\big(1+ \norm{U_{m,n}(s)}_2^2\big)\, ds \right].
\label{eq.comp-Poisson}
\end{align}
By choosing $\kappa_1=\kappa_2:=\tfrac{1}{4}$ 
and applying \eqref{eq.kappa1} - \eqref{eq.comp-Poisson} to inequality
\eqref{eq 2.16} we obtain for each $r\in [0,T]$ that 
\begin{align}\label{eq.inequality-Step1}
&\tfrac{1}{2} E\left[\sup_{t\in[0,r]}\norm{U_{m,n}(t)}_2^2\right] +E\left[\int_0^r \norm{U_{m,n}(s)}_{2,1 }^2\,ds\right]  + c_4 E\left[ \int_0^r \norm{U_{m,n}(s)}_{L^q }^q\,ds\right] \notag \\
&\qquad \le \norm{u_0}_2^2 + 2\delta_r^2 E \left[\int_0^r\norm{U_{m,n-1}(s)}^2_{2,1}\, ds\right]+a+a E \left[\int_0^r\norm{U_{m,n-1}(s)}^2_2ds\right]\notag \\
&\qquad\qquad  +  \int_0^r h(s) \left(1+ E\left[\norm{U_{m,n}(s)}_2^2\right]\right)\, ds,
\end{align}
where $h(s):=2h_2(s) + (1+c_{\kappa_1})h_4(s) + (1+c_{\kappa_2})h_6(s)$. 
Define the functions $\alpha_N^m\colon [0,T]\to \R$ and $\beta_N^m \colon [0,T]\to \R$ by 
$$
\alpha^m_N(r):=\sup_{n\leq N} E\left[\sup_{t\in[0,r]}\norm{U_{m,n}(t)}^2_2\right],\qquad \beta^m_N(r):=\sup_{n\leq N} E\left[\int_0^r\norm{U_{m,n}(t)}^2_{2,1 }dt\right].
$$
Inequality \eqref{eq.inequality-Step1} implies  for each $r\in [0,T]$:
\begin{align*}
\tfrac{1}{2} \alpha^m_N(r)+\beta^m_N(r)
\le \norm{u_0}^2 + 2\delta_r^2  \beta_N^m(r) + a+ \int_0^r h(s)\,ds + 
\int_0^r \big(a+h(s)\big)\alpha_N^m (s)\,ds. 
\end{align*}
By choosing $T_0\in [0,T]$ such that $2\delta_{T_0}^2=1/2$, 
we conclude from Gronwall's inequality that 
\begin{align*}
E\left[\sup_{t\in[0,T_0]}\norm{U_{m,n}(t)}_2^2 + \int_0^{T_0} \norm{U_{m,n}(s)}_{2,1 }^2\,ds\right]
\le  C
\end{align*}
for a constant $C>0$. Applying this to inequality \eqref{eq.inequality-Step1} completes the 
proof of \eqref{eq u mn}. 

\noindent
{\bf Step 2. } For $m\in\N$ and $n_1$, $n_2\in \N$ define 
$
\Gamma^m_{n_1,n_2}(t):= U_{m,n_1}(t)-U_{m,n_2}(t).
$
By similar arguments as in Step 1 and by Gronwall's lemma we obtain
\begin{align*}
\lim_{n_1,n_2\rightarrow +\infty}\left\{E\left[\sup_{t\in[0,T_0]}\norm{\Gamma^m_{n_1,n_2}(t)}^2_2\right]
                                       +
                                      E\left[\int_0^{T_0}\norm{\Gamma^m_{n_1,n_2}(t)}^2_{2,1 }dt\right]
\right\}
=0.
\end{align*}
Hence, there exists a adapted 
 process $U_m\in L^0\big(\Omega,\, D([0,T],L^2)\cap L^2([0,T],H^1_0 )\big)$ for each $m\in\N$ such that
$$
\lim_{n\rightarrow\infty} E\left[\sup_{s\in[0,T_0]}\norm{U_{m,n}(s)-U_m(s)}^2_2\right]
=
\lim_{n\rightarrow\infty} E \left[\int_0^{T_0}\norm{U_{m,n}(s)-U_m(s)}^2_{2,1 }ds\right]
=
0,
$$
and, due to \eqref{eq u mn}, satisfying
\begin{align}\label{eq u m}
E\left[\sup_{t\in[0,T_0]}\norm{U_m(t)}^2_2\right]
+
E\left[\int_0^{T_0}\norm{U_m(t)}^2_{2,1 }dt\right]
+
E\left[\int_0^{T_0}\norm{U_m(t)}^q_{L^q }dt\right]
\leq C.
\end{align}
Taking the limit in  (\ref{eq.Picard}) as $n\rightarrow \infty$ shows that $U_m$ is 
the unique solution of the equation 
\begin{align}\label{eq u m solution}
U_m(t)=u_0&+\int_0^t \bigg( \Delta U_m(s) + F\left(s, U_m(s)\right)\notag\\
&\qquad  + \div B\big(s,U_m(s)\big)+ \int_{-\infty}^0 \gamma(r)\Delta U_m(s+r)\, dr \bigg) \, ds \nonumber\\    
      &  +   \int_0^t G_1(s,U_m(s))\, dW(s) +\int_0^t\int_{K_m}G_2(s, U_m(s-),x)\,\widetilde{N}(dx,ds).
\end{align}

{\bf Step 3.} For $m$, $n\in\N$ with $n>m$ define $\Gamma_{n,m}(t):=U_n(t)-U_m(t)$. 
%which satisfies
%\begin{align*}
%\Gamma_{n,m}(t)=&\int_0^t \Delta \Gamma_{n,m}(s)ds + \int_{0}^t r(s, U_n(s))-r(s,U_m(s))ds \\
%                 &+ \int_0^t F(s, U_n(s))-F(s, U_m(s))ds+
%                 \int_0^t G_1(s,U_n(s))-G_1(s,U_m(s))dW(s)\\
%                 &+ \int_0^t\int_{K_m}G_2(s, U_n(s-),x)-G_2(s, U_m(s-),x)\widetilde{N}(dx,ds)\\
%                 &+\int_0^t\int_{K_n\setminus K_m}G_2(s, U_n(s-),x)\widetilde{N}(dx,ds).
%\end{align*}
Applying It\^o's formula and similar arguments as in Step 1 result in 
\begin{align}\label{eq.Ito-Step3}
&\norm{\Gamma_{n,m}(t)}^2_2+(1-2\delta_t^2)\int_0^t\norm{\Gamma_{n,m}(s)}^2_{2,1 }ds\notag \\
&\qquad \leq
c\int_0^t\norm{\Gamma_{n,m}(s)}^2_2ds + 2\int_0^t \langle \Gamma_{n,m}(s), G_1(s,U_n(s))-G_1(s,U_m(s))\rangle \,dW(s) \notag \\
&\qquad\qquad  +2\int_0^t\int_{K_m} \langle \Gamma_{n,m}(s-),G_2(s, U_n(s-),x)-G_2(s, U_m(s-),x)\rangle \,\widetilde{N}(dx,ds) \notag  \\
&\qquad\qquad  +\int_0^t\int_{K_m} \norm{G_2(s, U_n(s-),x)-G_2(s, U_m(s-),x)}^2_2 \,N(dx,ds)\notag  \\
&\qquad\qquad  +2\int_0^t\int_{K_n\setminus K_m} \langle \Gamma_{n,m}(s-),G_2(s, U_n(s-),x)\rangle \,\widetilde{N}(dx,ds)\notag  \\
&\qquad\qquad  +\int_0^t\int_{K_n\setminus K_m} \norm{G_2(s, U_n(s-),x)}^2_2 \,N(dx,ds).
\end{align}
It follows from  \eqref{eq.sigma_1-lipschitz} by Burkholder's inequality and Young's inequality that for each 
$\kappa_1\in (0,1)$ there exists a constant $c_{\kappa_1}>0$ such that 
\begin{align}\label{eq.Burkholder-Step3}
&E \left(\sup_{t\in[0,l]}\Big|2\int_0^t \langle \Gamma_{n,m}(s), G_1(s,U_n(s))-G_1(s,U_m(s))\rangle dW(s)\Big|\right)\nonumber\\
&\quad \leq c E\left(\int_0^l \norm{\Gamma_{n,m}(s)}^2_2\norm{G_1(s,U_n(s))-G_1(s,U_m(s))}^2_{L^2(H,L^2)}ds\right)^{1/2}\nonumber\\
&\quad \leq \kappa_1 E\left(\sup_{t\in[0,l]}\norm{\Gamma_{n,m}(t)}^2_2\right)
            +
            C_{\kappa_1}E\left(\int_0^l\norm{\Gamma_{n,m}(s)}^2_2ds\right),
\end{align}
In the same way it follows from \eqref{eq.sigma_2-lipschitz} that for each $\kappa_2$, $\kappa_3\in (0,1)$ there exist 
constants $c_{\kappa_2}$, $c_{\kappa_3}>0$ such that
\begin{align}
&E \left(\sup_{t\in[0,l]}\left|2\int_0^t\int_{K_m} \langle \Gamma_{n,m}(s-),G_2(s, U_n(s-),x)-G_2(s, U_m(s-),x)\rangle\,  \widetilde{N}(dx,ds)\right|\right)\nonumber\\
%&\quad \leq cE \left(\Big|\int_0^l\int_{K_m} \norm{\Gamma_{n,m}(s-)}^2_2 \norm{G_2(s, U_n(s-),x)-G_2(s, U_m(s-),x)}^2_2 N(dx,ds)\Big|^{1/2}\right)\nonumber\\
%&\quad \leq \kappa_2 E\left(\sup_{t\in[0,l]}\norm{\Gamma_{n,m}(t)}^2_2\right)
%            +
%            C_{\kappa_2} E \left(\int_0^l\int_{K_m}\norm{G_2(s, U_n(s),x)-G_2(s, U_m(s),x)}^2_2 \nu(dx)ds\right)\nonumber\\
&\quad \leq \kappa_2 E\left(\sup_{t\in[0,l]}\norm{\Gamma_{n,m}(t)}^2_2\right)
            +
            c_{\kappa_2} E \left(\int_0^l\int_{K_m}h_5^2(s,x)\norm{\Gamma_{n,m}(s)}^2_2 \,\nu(dx)ds\right),
\end{align}
and from \eqref{eq.sigma_2-growth} that
\begin{align}
&E \left(\sup_{t\in[0,l]}\abs{2\int_0^t\int_{K_n\setminus K_m} \langle \Gamma_{n,m}(s-),G_2(s, U_n(s-),x)\rangle \widetilde{N}(dx,ds)}\right) \\
%&\quad \leq \kappa_3 E\left(\sup_{t\in[0,l]}\norm{\Gamma_{n,m}(t)}^2_2\right)
%            +
%            C_{\kappa_3} E \left(\int_0^l\int_{K_n\setminus K_m}h_6^2(s,x)(\norm{U_n(s)}^2_2+1) \nu(dx)ds\right)\nonumber\\
&\quad \leq \kappa_3 E\left(\sup_{t\in[0,l]}\norm{\Gamma_{n,m}(t)}^2_2\right)
            +
            c_{\kappa_3}E \left(\sup_{s\in[0,l]}\left(\norm{U_n(s)}^2_2+1\right)\right)\int_0^l\int_{K_n\setminus K_m}h_6^2(s,x)\,\nu(dx)ds.\notag
\end{align}
Another application of \eqref{eq.sigma_2-lipschitz} and \eqref{eq.sigma_2-growth} imply 
\begin{align}
&E \left(\int_0^t\int_{K_m} \norm{G_2(s, U_n(s-),x)-G_2(s, U_m(s-),x)}^2_2 \, N(dx,ds)\right)\nonumber\\
&\qquad \qquad \leq E \left(\int_0^t\int_{K_m}h_5^2(s,x)\norm{\Gamma_{n,m}(s)}^2_2 \, \nu(dx)ds\right),
\intertext{and}
&E \left(\int_0^t\int_{K_n\setminus K_m} \norm{G_2(s, U_n(s-),x)}^2_2 \, N(dx,ds)\right)\nonumber\\
&\qquad \qquad \leq E \left(\sup_{s\in[0,t]}\left(\norm{U_n(s)}^2_2+1\right)\right)\int_0^t\int_{K_n\setminus K_m}h_6^2(s,x)\,\nu(dx)ds.\label{eq.H4-Step3}
\end{align}
Define the functions $\alpha_{n,m}\colon [0,T]\to \R$ and $\beta_{n,m}\colon [0,T]\to \R$ by 
$$
\alpha_{n,m}(l):=E\left(\sup_{t\in[0,l]}\norm{\Gamma_{n,m}(t)}^2_2\right),\qquad  \beta_{n,m}(l):=E\left(\int_0^l\norm{\Gamma_{n,m}(s)}^2_{2,1 }ds\right).
$$
By choosing $\kappa_1=\kappa_2=\kappa_3=1/6$ and recalling $2\delta^2_{T_0}=\tfrac{1}{2}$, we obtain by applying 
\eqref{eq.Burkholder-Step3} -- \eqref{eq.H4-Step3} to the inequality \eqref{eq.Ito-Step3} that 
\begin{align*}
& \tfrac{1}{2}\alpha_{n,m}(T_0)+ \tfrac{1}{2} \beta_{n,m}(T_0)\nonumber\\
&\qquad \leq
       \int_0^{T_0}\left(c+c_{\kappa_1}+(1+ c_{\kappa_2})\int_{K_m}h_5^2(s,x)\, \nu(dx)\right)\alpha_{n,m}(s)\, ds\nonumber\\
&\qquad\qquad 
       +
      (1+c_{\kappa_3})E \left(\sup_{s\in[0,T_0]}\left(\norm{U_n(s)}^2_2+1\right)\right)\int_0^{T_0}\int_{K_n\setminus K_m}h_6^2(s,x)\nu(dx)ds.
\end{align*}
Applying Gronwall's inequality and using $\int_0^{T_0}\int_{K_n\setminus K_m}h_6^2(s,x)\nu(dx)ds \to 0$
as $m,n\to \infty$ together with \eqref{eq u m} implies
\begin{align*}
 \lim_{n,m\rightarrow\infty} \alpha_{n,m}(T_0)+\beta_{n,m}(T_0)=0.
\end{align*}
Hence, there exists an $\mathbb{F}$-adapted process $U\in L^0\big(\Omega, \, D([0,T_0],L^2)\cap L^2([0,T_0],H^1_0 )\big)$ such that
$$
\lim_{n\rightarrow\infty} E\left[\sup_{s\in[0,T_0]}\norm{U_{m}(s)-U(s)}^2_2\right]
=
\lim_{n\rightarrow\infty} E \left[\int_0^{T_0}\norm{U_{m}(s)-U(s)}^2_{2,1 }ds\right]
=
0,
$$
which, due to \eqref{eq u m}, satisfies 
\begin{align}\label{eq u}
E\left[\sup_{t\in[0,T_0]}\norm{U(t)}^2_2\right]
+
E\left[\int_0^{T_0}\norm{U(t)}^2_{2,1 }dt\right]
+
E\left[\int_0^{T_0}\norm{U(t)}^q_{L^q }dt\right]
\leq C.
\end{align}
Taking limits in (\ref{eq u m solution}) shows that  that $U$ is the unique solution of \ref{eq heat} on the interval $[0,T_0]$.

{\textbf{Step 4.}} By repeating the above arguments we obtain the existence of a unique solution 
of \eqref{eq heat} on the interval $[T_0,2T_0]$ which finally leads to the completion of  the proof 
by further iterations. 
\end{proof}

\section{Large Deviation Principle}

%\subsection{\emph{Poisson Random Measure(PRM) and Brownian Motion(BM)}}\label{Section Representation}

%\subsubsection{\emph{PRM}}
%\label{Poisson Random Measure}

Recall that $H$ is a separable Hilbert space with an orthonormal basis $\{u_i\}_{i\in\mathbb{N}}$ and assume that $X$ is a locally compact Polish space with a $\sigma$-finite measure $\nu$ defined on $\Borel(X)$. 

Let $S$ be a locally compact Polish space. The space of all Borel measures on $S$ is denoted by $M(S)$ and the set of all  $\mu\in M(S)$ with $\mu(K)<\infty$
for each compact set $K\subseteq S$ is denoted by $M_{FC}(S)$. We endow $M_{FC}(S)$ with the weakest topology such that for each $f\in C_c(S)$ the mapping
$\mu\in M_{FC}(S)\rightarrow \int_Sf(s)\mu(ds)$ is continuous. This topology is metrizable such that $M_{FC}(S)$ is a Polish space, see \cite{Budhiraja-Dupuis-Maroulas.} for more details.
\vskip 0.2cm

In this section we specify the underlying probability space $(\Omega, \F, {\mathbb F}:=\{\F_t\}_{t\in [0,T]},P)$ in the following way: 
\begin{align*}
  \Omega:=C\big([0,T];H\big)\times M_{FC}\big([0,T]\times X
  \times [0,\infty)\big),\qquad \F:=\Borel(\Omega).
\end{align*}
We introduce the functions
\begin{align*}
&W\colon \Omega \rightarrow C\big([0,T];H\big),
\qquad W(\alpha,\beta)(t)=\sum_{i=1}^\infty{ \scapro{\alpha(t)}{u_i} u_i},\\
& N\colon \Omega \rightarrow M_{FC}\big([0,T]\times X\times [0,\infty)\big),\qquad  N(\alpha,\beta)=\beta.
\end{align*}
Define for each $t\in [0,T]$ the $\sigma$-algebra
\begin{align*}
\mathcal{G}_{t}:=\sigma\left(\left\{\big(W(s), \, N((0,s]\times A)\big):\,
0\leq s\leq t,\,A\in \mathcal{B}\big(X\times [0,\infty)\big)\right\}\right).
\end{align*}
For a given $\nu\in M_{FC}(X)$, it follows from \cite[Sec.I.8]{Ikeda-Watanabe} that there exists a unique probability measure $P$
 on $(\Omega,\mathcal{B}(\Omega))$ such that:
\begin{enumerate}
\item[(a)] $W$ is a cylindrical Brownian motion in $H$;
\item[(b)] $N$ is a Poisson random measure on $\Omega$ with intensity measure $\text{Leb}_T\otimes\nu\otimes \text{Leb}_\infty$;
\item[(c)] $W$ and $N$ are independent.
\end{enumerate}
We denote by $\mathbb{F}:=\{{\mathcal{F}}_{t}\}_{t\in[0,T]}$ the
$P$-completion of $\{\mathcal{G}_{t}\}_{t\in[0,T]}$ and by
$\mathcal P$ the $\mathbb{F}$-predictable $\sigma$-field
on $[0,T]\times \Omega$. 
Define
\begin{align*}
{\mathcal R}
:=\left\{\varphi\colon [0,T]\times X\times\Omega\to [0,\infty):
\, (\mathcal{P}\otimes\mathcal{B}(X))\setminus\mathcal{B}[0,\infty)\text{-measurable}\right\}.
\end{align*}
For $\varphi\in{\mathcal R}$, define a
counting process $N^{\varphi}$ on $[0,T]\times {X}$ by
   \begin{align*}%\label{Jump-representation}
      N^\varphi((0,t]\times A)(\cdot)=\int_{(0,t]\times A\times (0,\infty)}\1_{[0,\varphi(s,x,\cdot)]}(r)\, N(ds, dx, dr),
   \end{align*}
for $t\in[0,T]$ and $A\in\mathcal{B}(X)$.

For each $f\in L^2([0,T],H)$, we introduce the quantity
\begin{align*}
Q_{1}(f)
:=\frac{1}{2}\int_{0}^{T}\norm{f(s)}_{H}^{2}\,ds,
\end{align*}
and we define for each $m\in\N$ the space
\begin{align*}
 S_1^m:=\Big\{f\in L^{2}([0,T],H):\,Q_1(f)\leq m\Big\}.
\end{align*}
Equiped with the weak topology, $S_1^m$ is a compact subset of $L^2([0,T],H).$
We will throughout consider $S_1^m$ endowed with this topology.

By defining the function
\begin{align*}
\ell:[0,\infty)\rightarrow[0,\infty), \qquad
\ell(x)=x\log x-x +1
\end{align*}
we introduce for each measurable function
$g\colon [0,T]\times X\to [0,\infty)$ the quantity
\begin{align*}%\label{L_T}
Q_2(g):=\int_{[0,T]\times X}\ell\big(g(s,x)\big) \,ds \,\nu(dx).
\end{align*}
Define for each $m\in\N$ the space
\begin{align*}%\label{S_N}
     S_2^m:=\Big\{g:[0,T]\times X\rightarrow[0,\infty):\,Q_2(g)\leq m\Big\}.
\end{align*}
A function $g\in S_2^m$ can be identified with a measure $\hat{g}\in M_{FC}([0,T]\times X)$, defined by
   \begin{align}\label{eq.corres-func-meas}
      \hat{g}(A)=\int_A g(s,x)\,ds\,\nu(dx)\ \quad\text{ for all } A\in\mathcal{B}([0,T]\times X).
   \end{align}
This identification induces a topology on $S_2^m$ under which $S_2^m$ is a compact space, see the Appendix of \cite{Budhiraja-Chen-Dupuis}.
Throughout, we use this topology on $S_2^m$.\\

%By denoting the function
%\begin{align*}
%l:[0,\infty)\rightarrow[0,\infty), \qquad
%l(\varsigma)=\varsigma\log \varsigma-\varsigma+1
%\end{align*}
%we introduce for each $\varphi\in\bar{\mathbb{A}}$ the  random variable
%\begin{align}\label{L_T}
%L_T(\varphi)\colon \bar{V}\to [0,\infty], \qquad
% L_T(\varphi)(\omega):=\int_{[0,T]\timesX}l(\varphi(t,x,\omega))(Leb_T\otimes \nu)(dt,dx)
%\end{align}
%
%

%For each $\psi\in \mathcal{L}_2$,  define the random variable
%\begin{align*}
%\tilde{L}_{T}(\psi) \colon \bar{V}\to [0,\infty),\qquad
%\tilde{L}_{T}(\psi)(v)
%=\frac{1}{2}\int_{0}^{T}\norm{\psi(s,v)}_{U}^{2}\,ds.
%\end{align*}
%
%
%Set $\widetilde{\mathcal{U}}=\mathcal{L}_{2}\times \bar{\mathbb{A}}_b$. and
%$$\bar{L}_{T}(\varpi)=\tilde{L}_{T}(\psi)+L_{T}(\varphi)\text{ for }\varpi=(\psi,\varphi)\in \widetilde{\mathcal{U}}.$$

On the probability space $(\Omega,\Borel(\Omega),\F,{\mathbb F},P)$
carrying the cylindrical Brownian motion $W$ and the compensated Poisson random measure
$\widetilde{N}^{\epsilon^{-1}}(ds,dx):=N^{\epsilon^{-1}}(ds,dx)-\epsilon^{-1}\,\nu(dx)\,ds$ we consider for each $\epsilon>0$ the following stochastic heat equation
for $t\in [0,T]$:
\begin{align}\label{eq G}
U_\epsilon(t)&=
 u_0+\int_{0}^t \Big(\Delta U_\epsilon(s) + {\rm div}\, B\big(s,U_\epsilon(s)\big)
  +  F\big(s,U_\epsilon(s)\big)+\int_{-\infty}^0\gamma(r) \Delta U_\epsilon(s+r)\,dr\Big)\, ds\notag\\
&\qquad +\sqrt{\epsilon}\int_0^t G_1\big(s,U_\epsilon(s)\big)\,dW(s)+\epsilon\int_0^t\int_{X} G_2\big(s,U_\epsilon(s-),x\big)\,\widetilde{N}^{\epsilon^{-1}}(ds,dx),\\
U_\epsilon(s)&=\varrho(s)\text{ for }s<  0.  \notag
\end{align}
%where $\Phi(s,u)$ and $\Psi_\varrho(s)$ are defined as
%\begin{eqnarray*}
%\Phi(s,u):=\int_0^s \gamma (s-l) \Delta u(l)dl,\ \ \ \Psi_\varrho(s):=\int^0_{-\infty}\gamma(s-l)\Delta \varrho(l)dl.
%\end{eqnarray*}
The initial condition is given by  $(u_0,\varrho)\in {L^2}\times L^2(\R_{-},H_0^1)$.
Theorem \ref{thm solution} guarantees that for each $\epsilon>0$
there exists a unique solution $U_\epsilon:=U_\epsilon(u_0,\varrho)$ with trajectories in
the space $\mathcal{D}:=D([0,T],{L^2})\cap L^2([0,T],{H^1_0 })$.
\begin{theorem}
Under the assumption (H1)-(H4), the solutions $\{U_\epsilon:\, \epsilon>0\}$ of \eqref{eq G}  satisfy a large deviation principle on $\mathcal{D}$ with rate function $I\colon \mathcal{D}\to [0,\infty]$,
 where
\begin{align*}
I(\xi):=\inf\left\{Q_1(f)+Q_2(g):\, \xi=u(f,g),\,
f\in S_1^m, g\in S_2^m \text{ and }m\in\N \right\},
\end{align*}
and $u=u(f,g)\in\mathcal{D}$ solves the following deterministic partial differential
equation:
\begin{align}\label{eq u q}
u(t)&= u_0+\int_{0}^t \Big(\Delta u(s) + {\rm div}\, B\big(s,u(s)\big)
  +  F\big(s,u(s)\big)+\int_{-\infty}^0\gamma(r) \Delta u(s+r)\,dr \Big)\, ds \nonumber \\
&\qquad +\int_0^t G_1\big(s, u(s)\big)f(s)\,ds
  + \int_0^t\int_{X}G_2\big(s,u(s),x\big)\big(g(s,x)-1\big)\,\nu(dx)\,ds, \\
u(s)&=\varrho(s)\text{ for }s<0.\nonumber
\end{align}
\end{theorem}
\begin{proof}
Define the space $\mathcal{C}:=C([0,T];H)\times M_{FC}([0,T]\times X)$. Using the correspondence \eqref{eq.corres-func-meas},
 we define a function $\mathcal{G}^0\colon \mathcal{C}\to \mathcal{D}$ such that
\begin{align*}
\mathcal{G}^0\left(\int_0^{\cdot} f(s)\,ds,\,\hat{g}\right)=u(f,g)
\qquad\text{for all }f\in S_1^m,\, g\in S_2^m,\,m\in\N,
\end{align*}
where $u(f,g)$ is the unique solution of \eqref{eq u q}.
Theorem \ref{thm solution} implies that for
each $\epsilon>0$  there exists a mapping $\mathcal{G}^\epsilon\colon \mathcal{C}\to \mathcal{D} $ such that
\begin{align*}
\mathcal{G}^\epsilon\big(\sqrt{\epsilon} W,\epsilon N^{\epsilon^{-1}}\big)
\stackrel{\mathcal{D}}{=} U_\epsilon,
\end{align*}
where $U_\epsilon$ is the solution of (\ref{eq G})
(and $\stackrel{\mathcal{D}}{=}$ denotes equality in distribution).

Define for each $m\in\N$ a space of stochastic processes on $\Omega$ by
%\begin{eqnarray*}
%\mathcal{L}_{2}=\Big\{\psi\colon [0,T]\times
%\bar{\mathbb{V}}\to U:\text{measurable and} \; \int_{0}^{T}\norm{\psi(s,v)}_{U}^{2}\;ds<\infty,\\
%\text{for $\bar{\mathbb{P}}^{\bar{\mathbb{V}}}$-a.a. } v\in \bar{\mathbb{V}}
%\Big\},
%\end{eqnarray*}
%and introduce for each $N\in\N$ the space
\begin{align*}
\mathcal{S}_1^m:=\{\varphi\colon [0,T]\times
\Omega\to H:\, {\mathbb{F}}\text{-predictable and} \, \varphi(\cdot,\omega)\in S_1^m
\text{ for $P$-a.a. $\omega\in \Omega$}\}.
\end{align*}
Let $(K_n)_{n\in\N}$ be a sequence of compact sets $K_n\subseteq X$
with $ K_n \nearrow X$.  For each $n\in\N$, let
\begin{align*}
     \mathcal{R}_{b,n}
= \Big\{\psi\in \mathcal{R}:
\psi(t,x,\omega)\in \begin{cases}
 [\tfrac{1}{n},n], &\text{if }x\in K_n,\\
\{1\}, &\text{if }x\in K_n^c.
\end{cases}
\text{ for all }(t,\omega)\in [0,T]\times \Omega
\Big\},
\end{align*}
and let $\mathcal{R}_{b}=\bigcup _{n=1}^\infty \mathcal{R}_{b,n}$. Define
for each $m\in\N$ a space of stochastic process on $\Omega$ by
\begin{align*}
\mathcal{S}_2^m:=\{\psi\in  \mathcal{R}_{b}:\, \psi(\cdot,\cdot,\omega)\in S_2^m
\text{ for $P$-a.a. $\omega\in \Omega$}\}.
\end{align*}
According to Theorem 2.4 in \cite{Budhiraja-Chen-Dupuis} and Theorem 4.2 in
\cite{Budhiraja-Dupuis-Maroulas.}, our claim is established once we have proved:
\begin{enumerate}
\item[(C1)] if $(f_n)_{n\in\N}\subseteq S_1^m$ converges to $f\in S_1^m$
and $(g_n)_{n\in\N}\subseteq S_2^m$ converges to $g\in S_2^m$ for some $m\in\N$,
then
      $$
         \mathcal{G}^0\Big(\int_{0}^{\cdot}f_{n}(s)\,ds,\, \hat{g}_n\Big)\rightarrow \mathcal{G}^0\Big(\int_{0}^{\cdot}f(s)\, ds,
         \hat{g}\Big)\quad\text{in }\mathcal{D}.
      $$
\item[(C2)] if $(\phi_\epsilon)_{\epsilon>0}\subseteq \mathcal{S}_1^m$ converges weakly to $\phi\in \mathcal{S}_1^m$
and $(\psi_\epsilon)_{\epsilon> 0}\subseteq \mathcal{S}_2^m$ converges weakly to $\psi\in \mathcal{S}_2^m$ for some $m\in\N$,
then
    $$
         \mathcal{G}^\epsilon\Big(\sqrt{\epsilon} {W}+\int_{0}^{\cdot}\varphi_{\epsilon}(s)\, ds,\, \epsilon
          N^{\epsilon^{-1}\psi_\epsilon}\Big)\text{ converges weakly to }
         \mathcal{G}^0\Big(\int_{0}^{\cdot}\varphi(s)\,ds,\,\hat{\psi}\Big)\text{ in } \mathcal{D}.
      $$
\end{enumerate}
In the sequel, we will prove Condition (C2). The proof of Condition (C1)
follows analogously.
\end{proof}

\begin{lemma}
Let $(\phi_\epsilon)_{\epsilon>0}\subseteq \mathcal{S}_1^m$ and $(\psi_\epsilon)_{\epsilon> 0}\subseteq \mathcal{S}_2^m$ for some $m\in\N$. Then for each $\epsilon>0$ there exists a unique solution $V_\epsilon\in L^0\big(D([0,T],{L^2})\cap L^2([0,T],H^1_0 )\big)$ of
\begin{align}\label{eq hat u}
V_\epsilon(t)&= u_0+\int_{0}^t \Big(\Delta V_\epsilon(s) +
 {\rm div}\, B\big(s,V_\epsilon(s)\big)
  +  F\big(s,V_{\epsilon}(s)\big)+ \int_{-\infty}^0\gamma(r) \Delta V_\epsilon(s+r)\,dr \Big)\, ds\notag\\
&\qquad + \int_0^t G_1\big(s,V_\epsilon(s)\big) \varphi_\epsilon(s)\, ds + \sqrt{\epsilon}\int_0^t G_1\big(s, V_\epsilon(s)\big)\, dW(s)\\
&\qquad +\epsilon\int_0^t\int_{X}G_2\big(s,V_\epsilon(s-),x\big)\, \big(N^{\epsilon^{-1}\psi_\epsilon}(ds,dx)-\epsilon^{-1}\,ds\,\nu(dx)\big),\notag \\
V_\epsilon(s)&=\varrho(s)\text{ for }s<0.\notag
 \end{align}
Moreover, the solution $V_\epsilon$ has the same distribution  as
$\mathcal{G}^\epsilon\big(\sqrt{\epsilon} W+\int_0^{\cdot}\varphi_\epsilon(s)\,ds,\epsilon N^{\epsilon^{-1}\psi_\epsilon}\big)$.
\end{lemma}
\begin{proof}
The proof can be accomplished by following \cite[p.543]{Budhiraja-Chen-Dupuis}
and applying results in \cite[Se.A.2]{Budhiraja-Dupuis-Maroulas.}.
\end{proof}

\vskip 0.2cm

\begin{lemma}\label{lemma3.1}
For each $m\in\N$ there exists a constant $c_m>0$ such that for any $\epsilon\in(0,1)$:
\begin{align*}
 E\left (\sup_{t\in[0,T]}\norm{V_\epsilon(t)}_2^2 + \int_0^T\norm{V_\epsilon(t)}_{2,1 }^2\,dt + \int_0^T\norm{V_\epsilon(t)}_q^q\,dt \right)
 \leq c_m.
\end{align*}
\end{lemma}
\begin{proof}
Lemma 3.4 in \cite{Budhiraja-Chen-Dupuis} guarantees that for each $m\in\N$  there exists a constant $c_m>0$ such that
\begin{align}\label{eq estamite L}
\sup_{g \in S_2^m}\int_0^T\! \int_{X}\hspace*{-.2cm} h_i^2(s,x)g(s,x)\,\nu(dx)\,ds
+
\sup_{g\in S_2^m}\int_0^T\! \int_{X} \hspace*{-.2cm}  h_i(s,x)|g(s,x)-1|\,\nu(dx)\,ds
\leq c_m,
\end{align}
 for $i=5,6$. Using this inequality, the proof can be accomplished as the proof of \eqref{eq u mn}. 
\end{proof}

For each $\epsilon \ge 0$ let $Y_\epsilon$ be the unique solution of the SPDE
\begin{align*}
 Y_\epsilon(t)&=\int_0^t \Delta Y_\epsilon(s)\,ds + \sqrt{\epsilon}\int_0^t G_1(s, V_\epsilon(s))\,dW(s)\\ 
 &\qquad \qquad +\epsilon \int_0^t\int_{X}G_2(s,V_\epsilon(s-),x)\, \widetilde{N}^{\epsilon^{-1}\psi_\epsilon}(dx,ds)\quad \text{for all }t\in [0,T], 
\intertext{ and let $Z_\epsilon$ be the unique solution of the random PDE}
 Z_\epsilon(t)&=\int_0^t\Delta Z_\epsilon(s)\,ds + \int_0^t\int_{X}G_2(s,V_\epsilon(s),x)(\psi_{\epsilon}(s,x)-1)\,\nu(dx)\,ds \quad \text{for all }t\in [0,T].
\end{align*}
Furthermore, we define the difference
\begin{align*}
 J_\epsilon(t)=V_\epsilon(t)-Y_\epsilon(t)-Z_\epsilon(t)
 \qquad\text{for all }t\in [0,T].
\end{align*}
\begin{lemma}\label{lemma compact u}
For each $m\in\N$ there exists a constant $c_{m}$ independent of $\epsilon$ such that
\begin{itemize}
\item[{\rm (a)}] $Y_\epsilon$ obeys
$$
E\Big(\sup_{t\in[0,T]}\norm{Y_\epsilon(t)}_2^2\Big)+E\left(\int_0^T\norm{Y_\epsilon(t)}^2_{2,1}\,dt\right) \leq \epsilon\, c_{m};
$$
\item[{\rm (b)}] $\{Z_\epsilon, \epsilon \in (0,1)\}$ is tight in $C([0,T],L^2)$ and obeys
$$
E\Big(\sup_{t\in[0,T]}\norm{Z_\epsilon(t)}_2^2\Big)+E\left(\int_0^T\norm{Z_\epsilon(t)}^2_{2,1}\,dt\right) \leq c_{m};
$$
\item[{\rm (c)}] $\{J_\epsilon, \epsilon\in (0,1)\}$ is tight in $C([0,T],\mathbb{H}^{-d})\cap L^{2}([0,T],L^2)$
and obeys
\begin{align*}%\label{eq J}
E\Big(\sup_{t\in[0,T]}\norm{J_\epsilon(t)}^2_2\Big)+E\left(\int_0^T\norm{J_\epsilon(t)}^2_{2,1}\,dt\right) \leq c_{m}.
\end{align*}
\end{itemize}
\end{lemma}

\begin{proof} Part (a). It{\^o}'s formula implies 
\begin{align*}
   &\norm{Y_\epsilon(t)}_2^2+2\int_0^t\norm{Y_\epsilon(s)}^2_{2,1}\,ds\nonumber\\
&\qquad =
   2\sqrt{\epsilon}\int_0^t\langle G_1(s,V_\epsilon(s)), Y^\epsilon(s)\rangle \, dW(s)\\
&\qquad \qquad    +
   2\epsilon\int_0^t\int_{X}\langle G_2(s,V_\epsilon(s-),x), Y_\epsilon(s-)\rangle \,\widetilde{N}^{\epsilon^{-1}\psi_\epsilon}(dx,ds)\nonumber\\
   &\qquad\qquad +
   \epsilon\int_0^t \norm{G_1(s,V_\epsilon(s))}_{\mathcal{L}_2}^2\,ds
   +
   \epsilon^2\int_0^t\int_{X}\norm{G_2(s, V_\epsilon(s-),x)}^2_2 \,  N^{\epsilon^{-1}\psi_\epsilon}(dx,ds)\nonumber\\
   &\qquad =:
   I_1(t)+I_2(t)+I_3(t)+I_4(t).
\end{align*}
It follows from  \eqref{eq.sigma_1-growth} by Burkholder's inequality and Young's inequality that for each 
$\kappa_1\in (0,1)$ there exists a constant $c_{\kappa_1}>0$ such that 
\begin{align*}
   E\Big(\sup_{t\in[0,T]}\Big|I_1(t)\Big|\Big)
% &\leq
%    c E\Big(\Big|\int_0^T\epsilon \norm{Y^\epsilon(s)}^2_2 \norm{G_1(s, V_\epsilon(s))}^2_{\mathcal{L}_2}\,ds\Big|^{1/2}\Big)\nonumber\\
%&\leq
%    \kappa_1 E\Big(\sup_{t\in[0,T]}\norm{Y^\epsilon(t)}^2_{2}\Big)
%  +
%    \epsilon c_{\kappa_1}  E\Big(\int_0^T\norm{G_1(s,V_\epsilon(s))}^2_{\mathcal{L}_2}\,ds\Big)\nonumber\\
&\leq
   \kappa_1 E\Big(\sup_{t\in[0,T]}\norm{Y_\epsilon(t)}^2_2\Big)
   +
   \epsilon\, c_{\kappa_1}  E\Big(1+\sup_{t\in[0,T]}\norm{V_\epsilon(t)}_2^2\Big).
\end{align*}
Analogously, one obtains from \eqref{eq.sigma_2-growth} that for each $\kappa_2\in (0,1)$ there exists a constant $c_{\kappa_2}>0$ such that
\begin{align*}
  E\Big(\sup_{t\in[0,T]}\Big|I_2(t)\Big|\Big)
%&\qquad \leq
%  C E\Big(\Big|\int_0^T\int_{X}\epsilon^2\|Y^\epsilon(s-)\|_2^2\|G_2(s,V_\epsilon(s-),z)\|^2_HN^{\epsilon^{-1}\psi_\epsilon}(dz,ds)\Big|^{1/2}\Big)\nonumber\\
%&\qquad \leq
%  \kappa E\Big(\sup_{t\in[0,T]}\|Y^\epsilon(t)\|^2_2\Big)
%  +
%  C_\kappa \epsilon E\Big(\int_0^T\int_{X}\|G_2(s,V_\epsilon(s),z)\|^2_2\psi_{\epsilon}(s,z)\,\nu(dz)\,ds\Big)\nonumber\\
%&\qquad \leq
%  \kappa E\Big(\sup_{t\in[0,T]}\|Y^\epsilon(t)\|^2_2\Big)
%  +
%  C_\kappa \epsilon E\Big(\int_0^T\int_{X}h^2_6(s,z)(1+\|V_\epsilon(s)\|^2_2)\psi_\epsilon(s,z)\,\nu(dz)\,ds\Big)\nonumber\\
& \leq
  \kappa_2 E\Big(\sup_{t\in[0,T]}\norm{Y_\epsilon(t)}^2_2\Big)\\
  &\qquad +
  \epsilon\, c_{\kappa_2}  E\Big(1+\sup_{t\in[0,T]}\norm{V_\epsilon(t)}^2_2\Big)\cdot\Big(\sup_{g\in {S_2^m}}\int_{0}^T\int_{X}h^2_6(s,x)g(s,x)\,\nu(dx)\,ds\Big).
\end{align*}
Inequality \eqref{eq.sigma_1-growth} yields that there exits a constant $c_3>0$ such that 
\begin{align*}
 E\Big(\sup_{t\in[0,T]}|I_3(t)|\Big)
 \leq
 \epsilon\, c_3 E\Big(1+\sup_{t\in[0,T]}\norm{V_\epsilon(t)}^2_2\Big). 
\end{align*}
From inequality \eqref{eq.sigma_2-growth} we conclude that there  exits a constant $c_4>0$ such that 
\begin{align*}
 E\Big(\sup_{t\in[0,T]}|I_4(t)|\Big)
& \leq
 \epsilon E\Big(\int_0^T\int_{X}\norm{G_2(s,V_\epsilon(s),x)}^2_2\psi_\epsilon(s,x)\,\nu(dx)\,ds\Big)\nonumber\\
& \leq
  \epsilon c_4 E\Big(1+\sup_{t\in[0,T]}\norm{V_\epsilon(t)}_2^2\Big)\cdot \Big( \sup_{\chi\in {\mathcal{S}_2^m}}\int_0^T\int_{X}h^2_6(s,x)\chi(s,x)\,\nu(dx)\,ds\Big).
\end{align*}
Combining all of the above estimates and applying Gronwall's inequality together with 
Lemma \ref{lemma3.1} and \eqref{eq estamite L} completes the proof of part (a). 

\noindent Part (b).
By the chain rule we conclude from \eqref{eq.sigma_2-growth} and \eqref{eq estamite L} that there exists 
a constant $c>0$ with
\begin{align*}
  &\norm{Z_\epsilon(t)}^2_2+2\int_0^t\norm{Z_\epsilon(s)}^2_{2,1}\,ds\nonumber\\
  &\quad =
  \int_0^t\int_{X}\scaprob{G_2(s,V_\epsilon(s),x)(\psi_\epsilon(s,x)-1)}{Z_\epsilon(s)} \,\nu(dx)\,ds\nonumber\\
  &\quad \leq
  \int_0^t\int_{X}h_6(s,x)\big|\psi_\epsilon(s,x)-1\big|\big(1+\norm{V_\epsilon(s)}_2\big)\norm{Z_\epsilon(s)}_2\,\nu(dx)\,ds\nonumber\\
  &\quad \leq
 \left(\sup_{s\in[0,T]}\norm{Z_\epsilon(s)}_2\right) \left( \sup_{s\in[0,T]}(1+\norm{V_\epsilon(s)}_2)\right) \left(\sup_{g\in S_2^m}\int_0^T\int_{X} h_6(s,x)|g(s,x)-1|\,\nu(dx)\,ds\right)\nonumber\\
  &\quad \leq
  \tfrac{1}{2} \left(\sup_{s\in[0,T]}\norm{Z^\epsilon(s)}^2_2\right)
  +c\left(1+\sup_{s\in[0,T]}\norm{V_\epsilon(s)}^2_2\right),
\end{align*}
which completes the proof of the second claim due to Lemma \ref{lemma3.1}. 
Define for each  $M\ge 1$ and $m\in\mathbb{N}$ the set of functions
\begin{align*}
A_{M,m}&:=\bigg\{ \int_{X}G_2(\cdot, h(\cdot), x)(g(s,x)-1)\,\nu(dx):\\
& \qquad\qquad h\in D([0,T],L^2),\, \sup_{s\in[0,T]}\|h(s)\|^2_2\leq M,\, g\in {\mathcal{S}_2^m}\bigg\}.
\end{align*}
Proposition 3.9 of \cite{YZZ} guarantees that the set
\begin{align*}
 B_{M,m}=\left\{\int_0^\cdot e^{(\cdot-s)\Delta}f(s)\,ds,\ \ f\in A_{M,m} \right\}
\end{align*}
is relatively compact in $C([0,T],L^2)$. From Lemma \ref{lemma3.1} we conclude that 
\begin{align*}
P\Big(Z^\epsilon \in B_{M,m}\Big)
  \geq
  P\Big(\sup_{t\in[0,T]}\norm{V_\epsilon(t)}^2_2\leq M\Big)
%  & =
%  1-\mathbb{P}\Big(\sup_{t\in[0,T]}\|V_\epsilon(t)\|^2_2> M\Big)\nonumber\\
  \geq
  1-\frac{1}{M} E\Big(\sup_{t\in[0,T]}\norm{V_\epsilon(t)}^2_2\Big)
  &\geq
  1-\frac{c_{m}}{M},
\end{align*}
which completes the proof of part (b). 

\noindent Part (c). 
The very definition of $J_\epsilon$ yields
\begin{align*}
  J_\epsilon(t)
=
 u_0& +\int_0^t \left( \Delta J_\epsilon(s) + {\rm div} B(s,V_\epsilon(s))
 + F(s,V_\epsilon(s))+ \int_{-\infty}^0 \gamma(r)\Delta V_\epsilon(s+r)\,dr  \right)\,ds\nonumber\\
 &+ \int_0^t G_1(s,V_\epsilon(s))  \varphi_\epsilon(s)\,ds . 
\end{align*}
By combining Part (a), Part (b) and Lemma \ref{lemma3.1}  we  
conclude that for each $m\in\N$ there exists a constant $c_m>0$ such that 
\begin{align}\label{eq estamite J}
E\Big(\sup_{t\in[0,T]}\norm{J_\epsilon(t)}^2_2\Big)+E\Big(\int_0^T\norm{J_\epsilon(t)}^2_{2,1}\,dt\Big) \leq c_m.
\end{align}
Following the arguments in the proof of Lemma 4.3 in \cite{LXZ} by using \eqref{eq estamite J} and Lemma \ref{lemma3.1} implies that  there exist $\alpha\in(0,1)$  such that
for all $\epsilon\in(0,1)$ we have
\begin{align}\label{eq star 1}
E\bigg[\sup_{ \substack{ s,t\in[0,T]\\ s\neq t}  }\frac{\norm{J_\epsilon(s)-J_\epsilon(t)}_{2,-d}}{|s-t|^\alpha}\bigg]<\infty.
\end{align}
According to Lemma 4.3 in \cite{GRZ}, % or Lemma 2.1 in \cite{LXZ}
the estimates \eqref{eq estamite J} and \eqref{eq star 1} imply that
$\{J_\epsilon, \epsilon \in (0,1)\}$ is tight in $C([0,T],H ^{-d})\cap L^{2}([0,T],L^2)$,
which completes the proof of part (c).
\end{proof}

%\begin{proof}
\noindent
\textbf{{\bf Proof of Condition (C2)}}.\\
{\bf Step 1:}
We fix a sequence  $(\phi_\epsilon)_{\epsilon>0}\subseteq \mathcal{S}_1^m$ converging weakly to $\phi\in \mathcal{S}_1^m$  and $(\psi_\epsilon)_{\epsilon> 0}\subseteq \mathcal{S}_2^m$ converging weakly to $\psi\in \mathcal{S}_2^m$ for some $m\in\N$.   We conclude from Lemma  \ref{lemma compact u} that
$\big\{\big(Y_\epsilon, Z_\epsilon, J_\epsilon, \varphi_\epsilon,\psi_\epsilon, W, N\big),\, \epsilon\in(0,1)\big\}$  is  tight in the space
\begin{align*}
\Pi& :=\Big(
   D\big([0,T],{L^2}\big)\cap L^2\big([0,T],{H^1_0 }\big)\Big)\times C\big([0,T],{L^2}\big)\times \Big(C\big([0,T],H^{-d}\big)\cap L^{2}\big([0,T],{L^2}\big)\Big)\\
 &\qquad\qquad  \times \mathcal{S}_1^m\times \mathcal{S}_2^m\times  C\big([0,T],H\big)\times  M_{FC}\big([0,T]\times X
  \times [0,\infty)\big).
\end{align*}
Skorohod embedding theorem implies for any sequence  $(\epsilon_k)_{k\in\N} \subseteq\Rp$ converging to $0$, that there exist a probability space
$(\Omega^\prime, \mathcal{F}^\prime,  P^\prime)$, a sequence $\big\{\big(Y^\prime_k,Z^\prime_k,J^\prime_k, \varphi^\prime_k, \psi^\prime_k, W^\prime_k, N^\prime_k\big),$  $k\in\mathbb{N}\big\}$
as well as $\big(0, Z^\prime, J^\prime, \varphi^\prime,\psi^\prime, W^\prime, N^\prime\big)$ defined on this probability space and taking values in $\Pi$
such that:
\begin{align}
\big(Y^\prime_k,Z^\prime_k,J^\prime_k, \varphi^\prime_k, \psi^\prime_k, W^\prime_k, N^\prime_k\big)&\stackrel{\mathcal{D}}{=}(Y_{\epsilon_k}, Z_{\epsilon_k}, J_{\epsilon_k},  \varphi_k, \psi_k,  W, N)
\text{ for each $k\in\N$;}  \label{eq.equal-in-D}\\[.2cm]
\lim_{k\to\infty} \big(Y^\prime_k,Z^\prime_k,J^\prime_k, \varphi^\prime_k, \psi^\prime_k, W^\prime_k, N^\prime_k\big)&= \big(0, Z^\prime, J^\prime, \varphi^\prime,\psi^\prime, W^\prime, N^\prime\big)\quad\text{$P^\prime$-a.s.  in $\Pi$.}
  \label{eq.P-conv}
\end{align}
\begin{itemize}[wide, labelwidth=!, labelindent=0pt]
\item[(a)]
Lemma \ref{lemma compact u} together with \eqref{eq.equal-in-D} imply that
there exists a constant $c>0$ such that
\begin{align}
& E^\prime\Big(\sup_{t\in[0,T]}\|Y^\prime_k(t)\|^2_2\Big)
  + E^\prime\Big(\int_0^T\|Y^\prime_k(t)\|^2_{2,1}\,dt\Big)
 \leq  \epsilon_k c;\label{eq.E-Y_k}\\
& E^\prime\Big(\sup_{t\in[0,T]}\|Z^\prime_k(t)\|^2_2\Big)
   +E^\prime  \Big(\int_0^T\| Z^\prime_k(t)\|^2_{2,1}\,dt\Big)
  \leq  c;\label{eq.E-Z_k}\\
& E^\prime\Big(\sup_{t\in[0,T]}\|J^\prime_k(t)\|^2_2\Big)
  +E^\prime\Big(\int_0^T\|J^\prime_k(t)\|^2_{2,1}\,dt\Big)
 \leq c.  \label{eq.E-J_k}
\end{align}
Consequently, we conclude from \eqref{eq.P-conv} that
\begin{align}
& E^\prime\Big(\sup_{t\in[0,T]}\|\Z^\prime(t)\|^2_2\Big)
  + E^\prime\Big(\int_0^T\|Z^\prime(t)\|^2_{2,1}\,dt\Big)
  \leq c,\label{eq.E-Z_0}\\
&E^\prime\Big(\sup_{t\in[0,T]}\|J^\prime(t)\|^2_2\Big)
  + E^\prime\Big(\int_0^T\|J^\prime(t)\|^2_{2,1}\,dt\Big)
 \leq c; \label{eq.E-J_0}
\end{align}

\item[(b)] In the following we establish that
\begin{align}\label{eq.int-J_k-conv}
\lim_{k\rightarrow\infty}E^\prime\Big(\int_0^T\|J^\prime_k(t)-J^\prime(t)\|^2_2\, dt\Big)^{1/2}=0.
\end{align}
For this purpose, let $\delta>0$ be arbitrary and
define for each $k\in\N$ and $\omega^\prime\in\Omega^\prime$ the
set
$$
\Lambda_{k,\delta}(\omega^\prime):=\big\{s\in[0,T]:\ \|J^\prime_k(s,\omega^\prime)-J^\prime(s,\omega^\prime)\|_{2}\geq \delta\big\}.
$$
Since \eqref{eq.P-conv} implies
$\int_0^T\|J^\prime_k(t)-J^\prime(t)\|^{2}_2\,dt \to 0$ $P^\prime$-a.s. as 
$k\to\infty$ we obtain
\begin{align}\label{eq P24 star}
\text{Leb}\big(\Lambda_{k,\delta}(\cdot)\big)
\leq
\frac{1}{\delta^{2}} \int_0^T\|J^\prime_k(t)-J^\prime(t)\|^{2}_2 \, dt\to 0\quad   P^\prime\text{-a.s. as }k\rightarrow \infty.
\end{align}
By applying Fubini's theorem and Cauchy inequality, we derive
\begin{align*}
& E^\prime\left(\int_0^T\|J^\prime_k(t)-J^\prime(t)\|^{2}_2\, dt\right)^{1/2}\\
&\qquad\qquad  = \int_{\Omega^\prime} \Bigg(
     \int_{[0,T]\cap (\Lambda_{k,\delta}(\omega^\prime))^c}\|J^\prime_k(t,\omega^\prime)-  J^\prime(t,\omega^\prime)\|^2_2\,dt\\
&\qquad\qquad  \qquad\qquad     + \int_{[0,T]\cap \Lambda_{k,\delta}(\omega^\prime)}\|J^\prime_k(t,\omega^\prime)-J^\prime(t,\omega^\prime)\|^2_2\, dt
            \Bigg)^{1/2}\,P^\prime(d\omega^\prime)\\
&\qquad\qquad  \leq \delta T^{1/2}
     + E^\prime \left(\sup_{t\in[0,T]}\left(\|J^\prime_k(t)\|_2+\|J^\prime(t)\|_2\right)\, \left(\text{Leb}\big(\Lambda_{k,\delta}(\cdot)\big)\right)^{1/2}\right)\\
&\qquad\qquad  \leq
    \delta T^{1/2}
    +
    \left(E^\prime\left(\sup_{t\in[0,T]}\left(\|J^\prime_k(t)\|^2_2+\|J^\prime(t)\|^2_2\right)\right)\right)^{1/2}
    \Big(E^\prime\left(\text{Leb}\big(\Lambda_{k,\delta}(\cdot)\big) \right)\Big)^{1/2}.
\end{align*}
Since $\delta>0$ is arbitrary, we complete the proof of \eqref{eq.int-J_k-conv} by applying  Lebesgue's dominated convergence theorem together with
\eqref{eq P24 star} and using \eqref{eq.E-J_k} and \eqref{eq.E-J_0}.
 \item[(c)] Define for each $k\in\N$ the stochastic processes
 \begin{align*}
 S^\prime_k:= Z^\prime_k+ J^\prime_k,
\qquad S^\prime:= Z^\prime+ J^\prime,\qquad
 V^\prime_k:= Y^\prime_k+ Z^\prime_k+ J^\prime_k.
 \end{align*}
Note, that it follows from \eqref{eq.equal-in-D} and $V_{\epsilon}
=Y_\epsilon+Z_\epsilon+J_\epsilon$ that $ V^\prime_k$ has the same distribution as $V_{\epsilon_k}$.
The convergence \eqref{eq.P-conv} implies
\begin{align}
S^\prime_k\rightarrow S^\prime \qquad \text{ $ P^\prime$-a.s. in } L^2([0,T],{L^2}),\label{eq.conv-tilde-R}\\
 V^\prime_k\rightarrow S^\prime\qquad \text{ $ P^\prime$-a.s. in } L^2([0,T],{L^2}).\label{eq.conv-tilde-V}
\end{align}
Moreover, since $ Y^\prime_k\to 0$ $ P^\prime$-a.s. in the Skorkhod space $D([0,T],{L^2})$ according to \eqref{eq.P-conv}, we have
 $\sup_{t\in[0,T]}\| Y^\prime_k(t)\|_2\rightarrow 0$ $ P^\prime$-a.s.  Consequently,we obtain
\begin{align}
&\lim_{k\to\infty}\sup_{t\in[0,T]}\|S^\prime_k(t)-S^\prime(t)\|^2_{H^{-d}}= 0\quad \text{ $ P^\prime$-a.s.,}\label{eq.sup-conv-tilde-R}\\
&\lim_{k\to\infty} \sup_{t\in[0,T]}\| V^\prime_k(t)-S^\prime(t)\|^2_{H^{-d}}=0\ \quad \text{ $ P^\prime$-a.s.}\label{eq.sup-conv-tilde-V}
\end{align}
\item[(d)]
It follows from  Lemma \ref{lemma3.1} that there exists a stochastic process
$ V^\prime$ on  $\Omega^\prime$ such that
\begin{align*}
& V^\prime_k\rightarrow  V^\prime \qquad \text{weakly star in $L^2\Big(\Omega^\prime, L^\infty([0,T],{L^2})\Big)$,}\\
& V^\prime_k\rightarrow  V^\prime\qquad \text{weakly in $L^2\Big(\Omega^\prime, L^2([0,T],{H^1_0 })\Big)$,}\\
& V^\prime_k\rightarrow  V^\prime\qquad\text{weakly in $L^2\Big(\Omega^\prime, L^q([0,T],L^q )\Big)$,}\\
& V^\prime=S^\prime \qquad\text{$P^\prime\otimes $Leb-a.s.}
\end{align*}
In particular, it follows by \eqref{eq.E-Z_0}, \eqref{eq.E-J_0}
and Lemma \ref{lemma3.1} that
\begin{align}\label{eq ZJ estamate}
  E^\prime\left(\sup_{t\in[0,T]}\|S^\prime(t)\|_2^2\right)+ E^\prime\left(\int_0^T\|S^\prime(t)\|_{2,1 }^2 \, dt\right)+ E^\prime\left(\int_0^T\|S^\prime(t)\|_q^q \, dt\right)
 \leq C.
\end{align}
\end{itemize}

{\bf Step 2:} We will prove that $S^\prime$ solves equation \eqref{eq u q} for $f= \phi^\prime$ and $g= \psi^\prime$, that is
\begin{align}\label{eq P25 star}
\scapro{w}{S^\prime(t)}& = \scapro{w}{u_0}\notag \\
&\quad +\int_{0}^t \scaprobb{w}{\Delta S^\prime(s) + {\rm div}\, B\big(s,S^\prime(s)\big)
  +  F\big(s,S^\prime(s)\big)+\int_{-\infty}^0\gamma(r) \Delta S^\prime(s+r)\,dr }\, ds\notag\\
&\quad +\int_0^t\scaprobb{w}{G_1\big(s, S^\prime(s)\big)\varphi^\prime(s)}\,ds\\
&\quad
+\int_0^t\int_{X}\scaprobb{w}{G_2\big(s,S^\prime(s),z\big)\big( \psi^\prime(s,z)-1\big)}\,\nu(dz)\,ds, \notag\\
S^\prime(s)&=\varrho(s)\text{ for }s<0, \notag
\end{align}
for each  $w\in H_0^d$. For this purpose note that the very definition of
$S^\prime_k$ yields
\begin{align}\label{eq.aux1}
 & \scapro{w}{S^\prime_k(t)}-
\langle w, u_0\rangle \notag\\
&\qquad -
\int_0^t \scaprobb{w}{\Delta S^\prime_k(s)+ {\rm div}\, B(s,  V^\prime_k(s))
+F(s,  V^\prime_k(s))- \int_{-\infty}^0 \gamma(r)\Delta V^\prime_k(s+r)\, dr} \, ds \notag\\
 &\qquad  -     \int_0^t\scaprob{w}{ G_1(s,  V^\prime_k(s)) \phi^\prime_k(s)} \,ds\notag\\
&  =     \int_0^t\int_{X}\scaprob{w}{G_2(s, V^\prime_k(s), z)( \psi^\prime_k(s,z)-1)}
  \,\nu(dz)\, ds.
\end{align}
It follows from \eqref{eq.P-conv},  \eqref{eq.E-Z_k}, \eqref{eq.E-J_k}, \eqref{eq.conv-tilde-R}--\eqref{eq.sup-conv-tilde-V} by the same method as applied in  \cite[p.5234]{LXZ}, that there is a subsequence
%, which is also denoted by $(S^\prime_k)_{k\in\N}$,
such that
the left hand side of \eqref{eq.aux1} converges $P^\prime$-a.s. to
\begin{align*}%\label{eq 1}
 & \scapro{w}{S^\prime(t)}-
\langle w, u_0\rangle \notag\\
&\qquad +
\int_0^t \scaprob{w}{\Delta S^\prime(s)+ {\rm div}\, B(s, S^\prime(s))
+F(s, S^\prime(s))+ \int_{-\infty}^0 \gamma(r)\Delta S^\prime(s+r)\, dr} \, ds \notag\\
 &\qquad  +      \int_0^t\langle w, G_1(s, S^\prime(s))\phi^\prime(s)\rangle \, ds.
\end{align*}
For taking the limit on the right hand side in \eqref{eq.aux1}, we define for each $\delta>0$ and $\omega^\prime\in\Omega^\prime$ the set $\Lambda_{k,\delta}(\omega^\prime):=\big\{s\in[0,T],
\| V^\prime_k(s,\omega^\prime)-S^\prime(s,\omega^\prime)\|_2\geq \delta\big\}$ and
conclude  from \eqref{eq.conv-tilde-V} as in \eqref{eq P24 star} that
%\begin{align}\label{eq star 2}
$\text{Leb}(\Lambda_{k,\delta}(\cdot))\rightarrow 0$ $P^\prime$-a.s.
%\end{align}
Consequently, we obtain
\begin{align}\label{eq 2}
& E^\prime\Bigg(\Bigg|
  \int_0^t\int_{X}\scaprob{w}{G_2\big(s, V^\prime_k(s),x\big)( \psi^\prime_k(s,x)-1)} \,\nu(dx)\,ds\nonumber \\
&\quad\qquad\qquad  -
  \int_0^t\int_{X}\scaprob{w}{G_2\big(s, S^\prime(s),x\big)( \psi^\prime_k(s,x)-1)}\,\nu(dx)\,ds
\Bigg|\Bigg)\nonumber\\
&\quad \leq
  \|w\|_2 E^\prime\Bigg(\int_0^t\int_{X}\| V^\prime_k(s)-S^\prime(s)\|_2 h_5(s,x)| \psi^\prime_k(s,x)-1|\,\nu(dx)\,ds\Bigg)\nonumber\\
%&\quad \leq
%\|w\|_2  E^\prime\Bigg( \delta\sup_{g\in {\mathcal{S}_2^m}}\int_0^T\int_{X} h_5(s,x)|g(s,z)-1|\nu(dx)ds + \sup_{s\in[0,T]}\Big(\| V^\prime_k(s)\|_2+\|S^\prime(s)\|_2\Big) V_\delta\Bigg) \nonumber \\
\begin{split}
&\quad\le
\|w\|_2\Bigg(\delta\,\sup_{g\in S^m_2}\int_0^T\int_{X} h_5(s,x)|g(s,x)-1|\,\nu(dx)\,ds \\
&\qquad\qquad\qquad  +\Big( E^\prime\big( R_{\delta,k}^2 \big)\Big)^{1/2}\Big( E^\prime \Big(\sup_{s\in[0,T]}\Big(\| V^\prime_k(s)\|_2+\|S^\prime(s)\|_2\Big) \Big)^2\Big)^{1/2}\Bigg),
\end{split}
%&\quad \leq
%\delta \|w\|_2 C_{N,T}\nonumber\\
%       &\qquad +\left[E^1
%              \Big(\sup_{s\in[0,T]}\Big(\|\widetilde{u}^k(s)\|_2+\|\widetilde{ZJ}^0(s)\|_2\Big)\Big)^2
%            \right]^{1/2}
%                  \cdot
%                 \left[E^1
%              \Big(\sup_{\chi\in {\mathcal{S}_2^m}}\int_{[0,T]\cap \Lambda^\delta_k}\int_{X} h_5(s,z)|\chi(s,z)-1|\nu(dz)ds\Big)^2
%            \right]^{1/2}
%            \nonumber\\
%&\qquad \leq
%\delta \|w\|_2 C_{N,T}, \ \ \ \ as\ k\rightarrow\infty.
\end{align}
where
\begin{align*}
R_{\delta,k}:=\sup_{g\in S^m}\int_{[0,T]\cap \Lambda_{k,\delta}}\int_{X} h_5(s,x)|g(s,x)-1|\,\nu(dx)\,ds.
\end{align*}
Since Leb$(\Lambda_{k,\delta}(\cdot))\rightarrow 0$ $P^\prime$-a.s.\ as $k\to\infty$, we conclude from
 \cite[Le.3.4]{Budhiraja-Chen-Dupuis} or \cite[Eq.(3.19)]{YZZ} that
 $R_{\delta,k} \to 0 $ in $ P^\prime$-probability as $k\to \infty$. As $\delta>0$ is arbitrary, we  obtain from
\eqref{eq 2} together with \cite[Le.3.11]{Budhiraja-Chen-Dupuis}
  that the right hand side in \eqref{eq.aux1}
converges in $ P^\prime$-probability to
\begin{align*}%\label{eq 3}
  \int_0^t\int_{X} \scaprob{w}{G_2(s, S^\prime(s),x)( \psi^\prime(s,x)-1)}\, \nu(dx)\,ds,
\end{align*}
which varifies $S^\prime$ as a solution of \eqref{eq P25 star}.

\noindent
{\bf Step 3.}
In the last step we establish that $L_k:= V^\prime_k-S^\prime$ obeys
 \begin{align}\label{eq condition b}
 \sup_{t\in[0,T]}\|L_k(t)\|^2_2+\int_0^T\|L_k(t)\|^2_{2,1}\,dt\rightarrow 0\ \text{ in $ P^\prime$-probability for }\ k\rightarrow \infty.
\end{align}
Recall that $ V^\prime_k$ has the same distribution as $V_{\epsilon_k}$ and $S^\prime$ is a solution of \eqref{eq u q} according to Step 2. Thus, once we will have established \eqref{eq condition b}, it shows that  Condition (C2) is satisfied.

It{\^o}'s formula implies for each $t\in [0,T]$ that
\begin{align}\label{eq sum}
  \|L_k(t)\|^2_2=-2\int_0^t\|L_k(s)\|^2_{2,1}\,ds+ \sum_{i=1}^9 I_{k,i}(t),
\end{align}
where
\begin{align}
  \label{eq Ik1}
& I_{k,1}(t)=2\int_0^t\scaprobb{ L_k(s)}{ \int_{-s}^0\gamma(r)\Delta L_k(s+r)\, dr}\, ds,\\
   \label{eq Ik2}
& I_{k,2}(t)=2 \int_0^t \Big\langle L_k(s), {\rm div}\big(B(s,  V^\prime_k(s))-B(s,S^\prime(s))\big)\Big\rangle \, ds,\\
    \label{eq Ik3}
& I_{k,3}(t)=2 \int_0^t \scaprob{ L_k(s)}{ F(s, V^\prime_k(s))-F(s, S^\prime(s))} \, ds,\\
   \label{eq Ik4}
&  I_{k,4}(t)= 2 \int_0^t \scaprob{ L_k(s)}{G_1(s,  V^\prime_k(s))\varphi_k^\prime(s)-G_1(s, S^\prime(s))\phi^\prime(s)}\, ds,\\
  \label{eq Ik5}
  \begin{split}
& I_{k,5}(t)= 2 \int_0^t\int_{X}\scaprob{ L_k(s)}{ G_2(s, V^\prime_k(s),x)( \psi^\prime_k(s,x)-1)-G_2(s,S^\prime(s),x)\\
&\qquad\qquad\qquad  ( \psi^\prime(s,x)-1)}\, \nu(dx)\,ds,
 \end{split}\\
  \label{eq Ik6}
& I_{k,6}(t)=2\sqrt{\epsilon_k}\int_{0}^t \scaprob{L_k(s)}{G_1(s, V^\prime_k(s))}\, dW^\prime_k(s),\\
   \label{eq Ik7}
& I_{k,7}(t)=\epsilon_k\int_0^t\|G_1(s, V^\prime_k(s))\|^2_2\,
 ds,\\
   \label{eq Ik8}
& I_{k,8}(t)= 2\epsilon_k \int_0^t\int_{X}\scapro{L_k(s)}{ G_2(s,  V^\prime_k(s-),x)}\,  \widetilde{N}_k^{\prime^{\epsilon^{-1}\psi_k^\prime}}(dx,ds)\\
  \label{eq Ik9}
&  I_{k,9}(t)= \epsilon_k^2 \int_0^t\int_{X}\|G_2(s,  V^\prime_k(s),x)\|^2_2\, N_k^{\prime^{\epsilon^{-1}\psi^\prime_k}}(dx,ds).
\end{align}
By exploiting similar arguments as in \cite[p.5235]{LXZ}, we derive that
$ P^\prime$-a.s. we have
\begin{align}\label{eq Ik1 0}
 & I_{k,1}(t)\leq \left(1+\left(\int_{-t}^0|\gamma(r)|\, dr\right)^2\right)\int_0^t\|L_k(s)\|^2_{2,1}\, ds,\\
\label{eq Ik2 0}
& I_{k,2}(t)\leq \tfrac{1}{2}\int_{0}^t\|L_k(s)\|^2_{2,1}\, ds+c \int_0^t\|L_k(s)\|^2_2\, ds,\\
\label{eq Ik3 0}
& I_{k,3}(t)\leq c\int_0^t\|L_k(s)\|^2_2\,ds,\\
\label{eq Ik7 0}
& I_{k,7}(t)\leq c\epsilon_k \int_0^t\left(\| V^\prime_k(s)\|^2_2+1\right)\,ds
\end{align}
where $c>0$ denotes a generic constant.
Our Assumption (H3) yields
\begin{align}\label{eq Ik4 0}
&I_{k,4}(t)\nonumber\\\
& \leq
     2 \int_0^t\Big|\Big\langle L_k(s),\left(G_1(s, V^\prime_k(s))-G_1(s,S^\prime(s))\right)\phi^\prime_k(s)\Big\rangle\Big|\, ds\nonumber\\
&\qquad \qquad +
     2\int_0^t\Big|\Big\langle L_k(s), G_1(s,S^\prime(s))\left(\phi^\prime_k(s)-\phi^\prime(s)\right)\Big\rangle\Big|\, ds\nonumber\\
&\leq
     c\int_0^t\|L_k(s)\|^2_2\|\phi^\prime_k(s)\|_{H}\, ds
      +
     \left(1+\sup_{s\in[0,t]}\|S^\prime(s)\|_2\right)\int_0^t\|L_k(s)\|_2\left(\|\phi^\prime_k(s)\|_{H}+\|\phi^\prime(s)\|_{H}\right)\,ds\nonumber\\
&\leq
c\int_0^t\|L_k(s)\|_2^2\|\phi^\prime_k(s)\|_{H}\, ds + c\,m\left(1+\sup_{s\in[0,t]}\|S^\prime(s)\|_2\right) \left(\int_0^t\|L_k(s)\|_2^2\,ds\right)^{1/2}.
% c\,m\Big(\sup_{s\in[0,T]}\|\widetilde{u}^k(s)\|_2+\sup_{s\in[0,t]}\|S^\prime_2+1\Big)[\int_0^t\|L^k(s)\|^2_2]^{1/2},
\end{align}
By applying our assumption (H4) we obtain $ P^\prime$-a.s.\ that
\begin{align}\label{eq Ik5 0}
&|I_{k,5}(t)|\nonumber\\
&\leq
  2\int_0^t\int_{X}\|L_k(s)\|_2\|G_2(s, V^\prime_k(s),x)-G_2(s,S^\prime(s),x)\|_2| \psi^\prime_k(s,x)-1|\,\nu(dx)\, ds\nonumber\\
  &\qquad +
  2\int_0^t\int_{X}\|L_k(s)\|_2 \|G_2(s,S^\prime(s),x)\|_2\left(| \psi^\prime_k(s,x)-1|+| \psi^\prime(s,x)-1|\right)\,\nu(dx)\, ds \nonumber\\
& \leq
  2\int_0^t\int_{X}\|L_k(s)\|^2_2h_5(s,x)| \psi^\prime_k(s,x)-1|\,\nu(dx)\, ds+2\Big(1+\sup_{s\in[0,t]}\|S^\prime(s)\|_2\Big)K(t), 
%  &\qquad\qquad \times\,
%  \int_0^t\int_{X}\|L_k(s)\|_2 h_6(s,x)\Big(| \psi^\prime_k(s,x)-1|%+| \psi^\prime(s,x)-1|\Big)\,\nu(dx)\,ds.
\end{align}
where we use the notation
\begin{align*}
K(t):= \int_0^t\int_{X}\|L_k(s)\|_2h_6(s,x)\left(| \psi^\prime_k(s,x)-1|+| \psi^\prime(s,x)-1|\right)\,\nu(dx)\, ds .
\end{align*}
By applying the inequalities
(\ref{eq Ik1 0})--(\ref{eq Ik5 0}) to the equality \eqref{eq sum} we obtain $ P^\prime$-a.s.\  for each $t\in [0,T]$ that
\begin{align}\label{eq.gronwall-applied}
&  \|L_k(t)\|^2_2 + \left(\frac{1}{2} -\left(\int_{-t}^0|\gamma(r)|\,dr\right)^{2}\right)\int_0^t\|L_k(s)\|_{2,1}^2\, ds\notag\\
& \qquad \leq
     c\int_0^t\|L_k(s)\|^2_2 \left(2+\|\phi^\prime_k(s)\|_{H}+2\int_{X}h_5(s,x)| \psi^\prime_k(s,x)-1|\, \nu(dx)\right)\, ds\notag\\
&\qquad \qquad    +
       cm\left(1+\sup_{s\in[0,t]}\|S^\prime(s)\|_2\right)\left(\int_0^t\|L_k(s)\|^2_2\, ds\right)^{1/2}
       +
       c\epsilon_k\int_0^t \left(\| V^\prime_k(s)\|^2_2+1\right)\, ds
       \notag\\
&\qquad \qquad    +
     2\Big(1+\sup_{s\in[0,t]}\|S^\prime(s)\|_2\Big)K(t)+
     |I_{k,6}(t)|+|I_{k,8}(t)|+|I_{k,9}(t)|.
\end{align}
From \eqref{eq.conv-tilde-V} and (\ref{eq ZJ estamate})  we conclude
 that
\begin{align}\label{eq Jk1}
  J_{k,1}(t):=  cm\left(1+\sup_{s\in[0,t]}\|S^\prime(s)\|_2\right)\left(\int_0^t\|L_k(s)\|^2_2\, ds\right)^{1/2} \rightarrow 0 \quad  P^\prime\text{-a.s. }
\end{align}
as $k\to\infty$. The estimates  \eqref{eq.E-Y_k} to \eqref{eq.E-J_k}
results in
\begin{align}\label{eq Jk2}
J_{k,2}(t):=  c\epsilon_k\int_0^T\left(\| V^\prime_k(s)\|^2_2+1\right)\, ds
\rightarrow 0\qquad \text{ in }L^1(\Omega^\prime,\R)
\text{ as }k\to\infty.
\end{align}
In order to estimate the term
\begin{align*}
J_{k,3}(t) &:= 2\Big(1+\sup_{s\in[0,T]}\|S^\prime(s)\|_2\Big) K(t),
\end{align*}
we define for fixed $\delta>0$ and $\omega^\prime\in\Omega^\prime$ the set
$\Lambda_{k,\delta}(\omega^\prime):=\big\{s\in[0,T]:\ \|L_k(s,\omega^\prime)\|_2\geq \delta\big\}$.
It follows from \eqref{eq.conv-tilde-V} that
\begin{align*}
   \text{Leb}(\Lambda_{k,\delta}(\cdot))
 \leq
   \frac{1}{\delta^2}\int_0^T\|L_k(s)\|^2_2\,ds
   \rightarrow 0\quad   P^\prime\text{-a.s. as }k\rightarrow\infty.
\end{align*}
By splitting the integration domain we obtain
%We conclude from \eqref{eq estamite L} that
\begin{align*}
& E^\prime[ K_{k}(t)]\leq
   2\delta \sup_{g\in S_2^m}\int_0^t\int_{X}h_6(s,x)|g(s,x)-1|\,\nu(dx)\,ds+
    2E^\prime\left(\sup_{s\in[0,t]}\|L_k(s)\|_2 R_{\delta,k}(t)\right),
%\int_{[0,t]\cap \Lambda_{k,\delta}}\int_{X}h_6(s,x)| \psi^%\prime_k(s,x)-1|\,\nu(dx)\, ds\right)\\
%&\quad \leq
%    2\delta \sup_{g\in S_2^m}\int_0^t\int_{X}h_6(s,x)|g(s,x)-1|\,%\nu(dx)\,ds
  %  \\ &\qquad\qquad
%    +
%   \left( E^\prime\Big(\sup_{s\in[0,t]}\|L_k(s)\|^2_2\Big)
%   E^\prime\left(R^\prime_{\delta,k}(t) \right)^2\right)^{1/2},
\end{align*}
where we use the notation
\begin{align*}
R_{\delta,k}(t):=\sup_{g\in S_2^m}\int_{[0,t]\cap \Lambda_{k,\delta}}\int_{X}h_6(s,x)|g(s,x)-1|\,\nu(dx)\,ds.
\end{align*}
Since Leb$(\Lambda_{k,\delta}(\cdot))\rightarrow 0$ $P^\prime$-a.s.\ as $k\to\infty$, we conclude from
 \cite[Le.3.4]{Budhiraja-Chen-Dupuis} or \cite[Eq.(3.19)]{YZZ} that
 $R_{\delta,k}(t) \to 0 $ in $ P^\prime$-probability as $k\to \infty$. Thus, we obtain $E^\prime[K_k(t)]\to 0$ which results in 
\begin{align}\label{eq Jk3}
\lim_{k\to\infty}J_{k,3}(t)=0
\quad\text{ in $P^\prime$-probability.}
\end{align}
Choose $T_1>0$ such that $\int_{-T_1}^0|\gamma(r)|\,dr=\tfrac{1}{2}$.
Then by applying  Gronwall's inequality to \eqref{eq.gronwall-applied}
we obtain 
\begin{align}\label{eq star}
  &\sup_{t\in[0,T_1]}\|L_k(t)\|^2_2 + \tfrac{1}{4}\int_0^{T_1}\|L_k(s)\|_{2,1}^2\, ds  \notag  \\
 \begin{split}
  &\qquad\qquad \leq
     \Big(
           J_{k,1}(T_1)+J_{k,2}(T_1)+J_{k,3}(T_1)\\
     &\qquad\qquad  \qquad       +
          \sup_{t\in[0,T_1]}|I_{k,6}(t)|+\sup_{t\in[0,T_1]}|I_{k,8}(t)|+\sup_{t\in[0,T_1]}|I_{k,9}(t)|
     \Big)
     K_{k}(T_1),
\end{split}
\end{align}
where we define
$$
K_{k}(T_1):=\exp\left(c\int_0^{T_1}\left(2+\| \phi^\prime_k(s)\|_{H}+2\int_{X}h_5(s,x)| \psi^\prime_k(s,x)-1|\,\nu(dx)\right)\, ds\right).
$$
By the definition of $\mathcal{S}^m_1$ and \eqref{eq estamite L}   we have $K_k(T_1)\le c$ $ P^\prime$-a.s.\ for all $k\in\N$ for a generic constant $c>0$. By applying Burkholder's inequality to $I_{k,6}$ and $I_{k,8}$ we obtain
\begin{align}\label{eq Ik6 0}
 & E^\prime\Big(\sup_{t\in[0,T_1]}|I_{k,6}(t)|\Big)\nonumber\\
 &\qquad \leq
 2\sqrt{\epsilon_k}  E^\prime\Big(\int_0^{T_1}\|L_k(s)\|^2_2\|G_1(s, V^\prime_k(s))\|^2_2\,ds\Big)^{1/2}\nonumber\\
 &\qquad \leq
 2\sqrt{\epsilon_k}\Big( E^\prime\Big(\sup_{s\in[0,T_1]}\|L_k(s)\|^2_2\Big)+c E^\prime\Big(\int_0^{T_1}(\| V^\prime_k(s)\|^2_2+1)\,ds\Big)\Big),
\end{align}
and
\begin{align}\label{eq Ik8 0}
    & E^\prime\Big(\sup_{t\in[0,T_1]}|I_{k,8}(s)|\Big)\nonumber\\
 &\qquad \leq
    2 E^\prime\Big(\int_0^{T_1}\int_{X} \epsilon_k^2\|L_k(s)\|^2_2\|G_2(s, V^\prime_k(s),x)\|^2_2 \,N^{\prime^{\epsilon_k^{-1} \psi^\prime_k}}(dx,ds)\Big)^{1/2}\nonumber\\
 &\qquad \leq
     2  E^\prime\left(\epsilon_k^{1/4}\sup_{t\in[0,T_1]}\|L_k(t)\|_2
                     \Big(\int_0^{T_1}\int_{X}\epsilon_k^{3/2}\|G_2(s, V^\prime_k(s),x)\|^2_2\, N^{\prime^{\epsilon_k^{-1} \psi^\prime_k}}(dx,ds)\Big)^{1/2}
                \right)\nonumber\\
 &\qquad \leq
     \epsilon_k^{1/2} E^\prime\Big(\sup_{t\in[0,T_1]}\|L_k(t)\|^2_2\Big)
       +
     \epsilon_k^{1/2}c  E^\prime\Big(\int_0^{T_1}\int_{X}(\| V^\prime_k(s)\|^2_2+1)h_6^2(s,x) \psi^\prime_k(s,x)\,\nu(dx)\, ds\Big)\nonumber \\
 &\qquad \leq
     \epsilon_k^{1/2}  E^\prime\Big(\sup_{t\in[0,T_1]}\|L_k(t)\|^2_2\Big)
       +
     \epsilon_k^{1/2}c\, h  E^\prime\Big(\sup_{t\in[0,T_1]}\| V^\prime_k(t)\|^2_2+1\Big),
\end{align}
where we define
\begin{align*}
h:=\sup_{g\in S^m_2}\int_0^{T_1}\int_{X}h_6^2(s,x)g(s,x)\,\nu(dx)\,ds.
\end{align*}
Similarly, it follows from \eqref{eq.sigma_2-growth} that
\begin{align}\label{eq Ik9 0}
 E^\prime(I_{k,9}(T_1))
\leq
\epsilon_k c \left(\sup_{g\in S_2^N}\int_0^{T_1}\int_{X}h_6^2(s,x) g(s,x)\,\nu(dx)\, ds\right)  E^\prime\Big(\sup_{t\in[0,T_1]}\| V^\prime_k(t)\|^2_2+1\Big).
\end{align}
By applying \eqref{eq Jk1} -- \eqref{eq Jk3} and \eqref{eq Ik6 0} --\eqref{eq Ik9 0} to inequality \eqref{eq star} we  conclude 
$$
\sup_{t\in[0,T_1]}\|L_k(t)\|^2_2+\int_0^{T_1}\|L_k(t)\|^2_{2,1}\,dt\rightarrow 0\text{ in $ P^\prime$- probability}.
$$
As the definition of $T_1$ only depends on $\gamma$, we can repeat the above procedure 
for the interval $[T_1,2T_1]$, which after further iterations finally leads to the proof of \eqref{eq condition b}.
%\end{proof}

%\bibliographystyle{plain}
%\bibliography{Jianliang}

%\end{document}

\end{document}